\documentclass[12pt]{amsart}

\usepackage{lineno,hyperref}
\usepackage{amsmath,amsfonts,amsthm,mathrsfs,mathtools,amscd}
\usepackage{amssymb,latexsym}
\usepackage{cases}
\usepackage[numbers]{natbib}
\usepackage{booktabs}
\usepackage{bm} 
\usepackage{cleveref}
\usepackage{paralist}
\usepackage{multirow}
\usepackage{enumerate}
\usepackage{adjustbox}
\usepackage[shortlabels]{enumitem}
\setlist[enumerate]{nosep}
\usepackage{array}
\usepackage{tikz-cd}
\usepackage{witharrows}

\usepackage[figuresright]{rotating}

\usepackage[section]{placeins}

\renewcommand\appendix{\setcounter{secnumdepth}{-2}}

\usepackage{url,cancel}
\usepackage{float}

\usepackage{caption}

\newtheorem{thm}{Theorem}[section]

\newtheorem{lem}[thm]{Lemma}
\newtheorem{prop}[thm]{Proposition}
\newtheorem{cor}[thm]{Corollary}
\theoremstyle{remark}
\newtheorem{re}[thm]{Remark}
\theoremstyle{definition}

\newtheorem{defn}[thm]{Definition}

\newcommand{\tr}{\mbox{Tr}}
 \newcommand{\nr}{\mbox{N}}
\newcommand{\ho}{\mbox{H}}

\newcommand{\vol}{\mathrm{vol}}
\newcommand{\diag}{\mathrm{diag}}

\makeatletter

\@addtoreset{equation}{section}
\makeatother
\numberwithin{equation}{section}

\begin{document}
\author{Zilong He}
\address{Department of Mathematics,	Dongguan University of Technology, Dongguan 523808, China}
\email{zilonghe@connect.hku.hk}
	
\title[]{Finiteness theorems for trace Euclidean quadratic lattices over totally real number fields}
\subjclass[2020]{Primary 11E12; Secondary 11H31, 11R80}
\date{}
\keywords{trace Euclidean lattice, quadratic form, finiteness, totally real number field}
\begin{abstract}
	Motivated by Clark's work on Euclidean quadratic forms, we introduce
	trace Euclidean and house Euclidean quadratic lattices over totally
	real number fields, and relate these concepts through the $p$-norm
	Euclidean property. We prove that, as $F$ varies over all totally
	real number fields of arbitrary degree, there are only finitely many 
	positive definite trace Euclidean $\mathcal{O}_{F}$-lattices
	of arbitrary rank $n\ge1$, up to isometry. For every $p\in[1,\infty]$,
	the same finiteness holds for positive definite $p$-norm
	Euclidean lattices. We also classify all positive definite 
	trace Euclidean lattices of rank one over real quadratic fields:
	there are exactly six isometry classes, and every such lattice is free. As a corollary, the trace Euclidean real quadratic fields are precisely
	$\mathbb{Q}(\sqrt{2})$, $\mathbb{Q}(\sqrt{5})$,
	and $\mathbb{Q}(\sqrt{13})$.
\end{abstract}
\maketitle

\section{Introduction}
\label{sec:introduction}

The classification of positive definite quadratic forms over $\mathbb{Z}$ satisfying prescribed representation properties has been studied extensively. Among these properties, universality is the most prominent. A positive definite quadratic form $Q$ over $\mathbb{Z}$ is called \textit{universal} if it represents all positive integers. Its study originated from Lagrange's four-square theorem: the quadratic form $x_{1}^{2}+x_{2}^{2}+x_{3}^{2}+x_{4}^{2}$ is universal. The systematic study of universal quadratic forms was initiated by Ramanujan and subsequently developed by Dickson and Willerding \cite{dickson_quaternary-universal_1927,dickson_universal_1929,ramanujan_universal-quaternary_1917,willerding_universal-quaternary_1948}. Finally, universality was characterized by the Conway--Schneeberger $15$-Theorem for classical integral quadratic forms \cite{conway_fifteen_2000}, and by the Bhargava--Hanke $290$-Theorem for integral quadratic forms \cite{bhargava-hanke_290_2011}; see Kim's survey \cite{kim_recent-universal_2004} for more historical background. Some properties weaker than universality have also been studied, such as almost universality for quaternary quadratic forms by Bochnak and Oh \cite{bochnak-oh_almost-universal_2009}, spinor regularity for ternary quadratic forms by Chan and Earnest \cite{chan-earnest_spinor-regular_2004}, and almost regularity for ternary quadratic forms by Chan and Oh \cite{chan-oh_almost-regular_2004}. When $\mathbb{Z}$ is replaced by rings of algebraic integers, the classification problem becomes substantially more difficult, even for universal quadratic forms, and consequently only a few complete classifications are known; see the work of Chan, Kim, and Raghavan \cite{chan-kim-raghavan_universal_1996} and Kala's survey \cite{kala_universal-indecomposables_2023}. Therefore, establishing finiteness or existence results is often the first step toward a complete classification of special quadratic forms. For results on regularity, see Chan et al. \cite{chan-earnest-icaza-kim_2007} and Chan and Icaza \cite{chan-icaza}; for results on universality, see Kala and collaborators \cite{kala-kramer-krasensky_kitaoka-squares_2026,kala-yatsyna_kitaoka-lifting_2023,kala-yatsyna_even-better-squares_2026,kala-yatsyna-zmija_density-zero}, Kim et al. \cite{kim-kim-park_rank-seven_2022}, and Kr\'asensk\'y et al. \cite{krasensky-tinkova-zemkova_biquadratic_2020}, as well as the references therein. 

In \cite{clark_ADC-I-2012}, Clark introduced Euclidean quadratic forms over $\mathbb{Z}$: a positive definite $n$-ary quadratic form $Q$ is called \textit{Euclidean} if for every $x\in \mathbb{Q}^{n}$, there exists $y\in \mathbb{Z}^{n}$ such that $Q(x-y)<1$. To extend this notion to the ring of integers $\mathcal{O}_{F}$ of a totally real number field $F$, it is natural to equip $\mathcal{O}_{F}$ with the absolute norm $|\nr_{F/\mathbb{Q}}(\cdot)|$. Let $V$ be a non-degenerate quadratic space over $F$ with a quadratic form $Q$, and let $L$ be a positive definite $\mathcal{O}_{F}$-lattice on $V$. We say that $L$ is \textit{norm Euclidean} (or simply \textit{Euclidean}) if for every $x\in V$, there exists $y\in L$ such that $\nr_{F/\mathbb{Q}}(Q(x-y))<1$. If $L$ is free and norm Euclidean, we also call the associated quadratic form $Q$ Euclidean. In \cite[Conjecture 29]{clark_ADC-I-2012}, Clark further conjectured that, up to isometry, there are only finitely many totally definite Euclidean quadratic forms $Q$ over $F$, as $F$ varies over all totally real number fields. Over $\mathbb{Q}$, Nebe's classification, as corrected and independently reclassified by Clark and Jagy, gives exactly $70$ positive definite Euclidean $\mathbb{Z}$-lattices up to isometry; see \cite[Theorem~1.6, Remark 1.7, and \S 6.4]{clark_ADC-II-2014}. However, the conjecture remains open in general, and it is even unknown whether there are only finitely many isometry classes of positive definite Euclidean $\mathcal{O}_{F}$-lattices over a fixed totally real number field $F \neq \mathbb{Q}$. We will further discuss this conjecture in a separate paper \cite{he_norm-eculidean_2026}. 

Motivated by Clark's work, we define analogous notions by replacing the norm $\nr_{F/\mathbb{Q}}$ with either the trace $\tr_{F/\mathbb{Q}}$ or the house $\ho_{F/\mathbb{Q}}$, and introduce the $p$-norm Euclidean property (cf. Definitions \ref{defn:trace-Euclidean} and \ref{defn:p-norm}) to relate these three properties. We prove finiteness theorems for positive definite trace Euclidean lattices of arbitrary rank $n\ge1$, uniformly as the totally real field $F$ and its degree $d$ vary, in both the classic integral case (Theorem \ref{thm:finiteness-classic}) and the integral case (Theorem \ref{thm:finiteness-integral}). Using Voight's minimal discriminants \cite{voight_totally-real_2008} and unconditional Odlyzko bounds, we also obtain explicit upper bounds on the rank and degree, as shown in Tables \ref{tab:admissible-classical}--\ref{tab:max-rank-degree-norm}. These results can be viewed as trace analogues of Clark's conjecture. Using \eqref{implication-Euclidean-maximal}, we deduce corresponding finiteness results for $p$-norm Euclidean lattices (Corollary \ref{cor:p-norm}). 

The trace Euclidean property is also related to the theory of ideal lattices and thin number fields. Bayer-Fluckiger proved finiteness of thin fields and of number fields containing a thin ideal class (cf. \cite[Propositions 11.4 and 12.4]{bayer-fluckiger_euclidean_2006}), and Bayer-Fluckiger and Nebe also listed all the candidates for totally real thin fields; see \cite[Theorem 5.1 and Table 1]{bayer-fluckiger-nebe_euclidean_2005}. In \cite[\S8]{dutour-schurmann-vallentin_voronoi_2008}, Dutour Sikiri\'c, Sch\"urmann, and Vallentin completed the classification
by excluding the three remaining candidates. In fact, in the rank one case, integrality and the trace Euclidean condition imply that the underlying fractional ideal is thin, as explained in Section \ref{sec:trace-euclidean-fields}. Applying the covering formula of Clark and Jagy \cite[Theorem 5.1(b)]{clark_ADC-II-2014}, we classify all positive definite trace Euclidean lattices of rank one over real quadratic fields (Theorem \ref{thm:rank-one-classification}). In particular, we determine all trace Euclidean real quadratic fields (Corollary \ref{cor:trace-euclidean-quadratic}).

Before proceeding to our results, we review the basic definitions and notation. Let $F$ be a totally real number field, and let $\mathcal{O}_{F}$ be the ring of integers of $F$. Let $d=[F:\mathbb{Q}]$ and $\sigma_{1},\ldots,\sigma_{d}$ be $d$ distinct real embeddings from $F$ into $ \mathbb{R}$. For $\alpha \in F$, we define the trace, norm and house of $\alpha$ by
\begin{align*}
	\tr_{F/\mathbb{Q}}(\alpha) = \sum_{i=1}^{d} \sigma_i(\alpha), \quad
	\nr_{F/\mathbb{Q}}(\alpha) = \prod_{i=1}^{d} \sigma_i(\alpha), \quad 
	\ho_{F/\mathbb{Q}}(\alpha) = \max_{1\le i\le d} |\sigma_i(\alpha)|,
\end{align*}
respectively. We denote by $\Delta_{F}$ the discriminant of $F$. As usual, we write $\nr(\mathfrak{a})=|\mathcal{O}_F/\mathfrak{a}|$ for the absolute norm of a nonzero integral ideal $\mathfrak{a}$. For a nonzero fractional ideal $\mathfrak{b}$ of $F$, we choose
$c\in F^{\times}$ such that $c\mathfrak{b}$ is integral, and put
\begin{align*}
	\nr(\mathfrak{b})
	\coloneq
	\frac{\nr(c\mathfrak{b})}
	{|\nr_{F/\mathbb{Q}}(c)|}.
\end{align*}
This definition is independent of the choice of $c$ 
(cf. \cite[Definition 9.4.1]
{alaca_williams_2004}). In particular, $
	\nr((\alpha))
	=
	|\nr_{F/\mathbb{Q}}(\alpha)|$
for the principal ideal $(\alpha)$ generated by the element $\alpha\in F^{\times}$; see
\cite[Theorems 9.2.5 and 9.4.1]
{alaca_williams_2004}.

Let $V$ be an $n$-dimensional quadratic space over $F$ with associated symmetric bilinear form $B$ and quadratic form $Q$ such that
\begin{align*}
	B(x,y) = \dfrac{1}{2}(Q(x+y)-Q(x)-Q(y)),\quad x,y\in V.
\end{align*} 
Let $L$ be an $\mathcal{O}_{F}$-lattice of rank $n$ on $V$. As usual, we denote the scale ideal, norm ideal, and volume ideal of $L$ by $\mathfrak{s}(L)$, $\mathfrak{n}(L)$, and $\mathfrak{v}(L)$, respectively. We call $L$ \emph{classic integral} if $\mathfrak{s}(L)\subseteq \mathcal{O}_F$ and \emph{integral} if $\mathfrak{n}(L)\subseteq \mathcal{O}_F$. Recall that $\mathfrak{n}(L)\subseteq \mathfrak{s}(L)$, so if $L$ is classic integral, then it is integral. We call $L$ \emph{$\mathcal{O}_{F}$-maximal} if there is no integral $\mathcal{O}_{F}$-lattice $L^{\prime}$ on $V$ such that $L\subsetneq L^{\prime}$. Throughout this paper, all lattices are assumed to be positive definite and integral.

\begin{defn}\label{defn:trace-Euclidean}
	Let $F$ be a totally real number field of degree $d$ and $L$ be a positive definite $\mathcal{O}_{F}$-lattice.
	\begin{enumerate}[label=\rm(\roman*), itemindent=0em, leftmargin=*]
		\item We say that $L$ is trace Euclidean over $F$ if for every $x \in V$, there exists $y \in L$ such that $\tr_{F/\mathbb{Q}}(Q(x-y)) < d$.
 
		\item We say that $L$ is house Euclidean over $F$ if for every $x \in V$, there exists $y \in L$ such that $\ho_{F/\mathbb{Q}}(Q(x-y)) < 1$.
	\end{enumerate}
\end{defn}
If $d=1$, i.e., $F=\mathbb{Q}$, then the trace Euclidean condition coincides with Clark's definition of a Euclidean lattice over $\mathbb{Z}$ in \cite[\S 2.1]{clark_ADC-I-2012}.
\begin{thm}\label{thm:finiteness-classic}
	As $F$ varies over all totally real number fields of arbitrary degree $d=[F:\mathbb{Q}]$, there are only finitely many positive definite trace Euclidean classic integral $\mathcal{O}_{F}$-lattices of arbitrary rank $n\ge1$, up to isometry. Every such lattice satisfies
	\begin{align*}
		n\Delta_{F}^{1/d}<2\pi e.
	\end{align*}
\end{thm}

\begin{thm}\label{thm:finiteness-integral}
	As $F$ varies over all totally real number fields of arbitrary degree $d=[F:\mathbb{Q}]$, there are only finitely many positive definite trace Euclidean integral $\mathcal{O}_{F}$-lattices of arbitrary rank $n\ge1$, up to isometry. Every such lattice satisfies
	\begin{align*}
		n\Delta_{F}^{1/d}<4\pi e.
	\end{align*}
\end{thm}

\begin{re}\label{re:finiteness}
	(i) In Theorems \ref{thm:finiteness-classic}
	and \ref{thm:finiteness-integral}, there are only finitely many isomorphism classes of totally real number fields $F$
	that admit a trace Euclidean lattice satisfying the corresponding
	integrality condition. For each fixed $F$,
	there are only finitely many isometry classes of such
	lattices over $F$, taken over all ranks $n\ge1$.
	
	(ii) Neither the degree $d$ nor the rank $n$ is fixed in advance:
	the finiteness assertions hold as both parameters vary simultaneously.
	In particular, they include finiteness for each fixed rank $n\ge1$
	with arbitrary field degree, and for each fixed degree $d\ge1$
	with arbitrary rank.
	
	(iii) The numerical bounds in
	Tables \ref{tab:admissible-classical}--\ref{tab:max-rank-degree-norm}
	are obtained from the necessary conditions
	\eqref{eq:admissible-pairs}.
	Using $\beta_d$ from \eqref{eq:beta-discriminant}, we compute
	$H(n,d)$ in \eqref{eq:Hd} and
	$h_{\mathfrak{s}}(n,d)$ and $h_{\mathfrak{n}}(n,d)$
	in \eqref{eq:def-h} with Mathematica, for every pair
	\begin{align*}
		1\le n\le34 \quad\text{and}\quad 1\le d\le14,
	\end{align*}
	as permitted by Corollary \ref{cor:finite-grid}. These admissible pairs, together with the corresponding upper
	bounds for $\nr(\mathfrak{v}(L))$ given by \eqref{eq:Hd},
	are listed in Tables \ref{tab:admissible-classical}
	and \ref{tab:admissible-norm}. The values of $h_{\mathfrak{s}}(n,d)$ and
	$h_{\mathfrak{n}}(n,d)$ are rounded to six decimal places,
	and the displayed upper bounds for $\nr(\mathfrak{v}(L))$
	are rounded upwards to six significant digits.
	Here admissibility means only that the pair satisfies the
	necessary conditions \eqref{eq:admissible-pairs}, but it does not
	assert the existence of a corresponding trace Euclidean lattice.
	Tables \ref{tab:max-rank-degree-classical}
	and \ref{tab:max-rank-degree-norm} summarize the resulting
	upper bounds $N_d$ for the rank at each fixed degree $d$
	and $D_n$ for the degree at each fixed rank $n$.
	In particular,
	\begin{align*}
		\begin{cases}
			n\le12,\;d\le6
			&\text{in the classic integral case},\\
				n\le29,\; d\le9
			&\text{in the integral case}.
		\end{cases} 
	\end{align*}
	In the integral case, the only admissible rank for $d=7,8,9$
	is $2$. Indeed, every integral lattice of rank one is
	classic integral, so in both cases,
	$H(1,d)\ge 1$ implies that $d\le6$.
	
	(iv) The bounds above use only \eqref{eq:admissible-pairs}.
	When $d=1$, the complete classification of Clark and Jagy
	\cite[\S6.4 and Table 3]{clark_ADC-II-2014} gives $n\le10$.
	For $d\ge2$, Table \ref{tab:max-rank-degree-norm} gives $n\le12$.
	Therefore, every positive definite integral trace Euclidean
	lattice satisfies $n\le12$.
\end{re}

Let $x \in V$ and $y \in L$. Write $c_i = \sigma_i(Q(x-y)) \ge 0$ for $i=1,\dots,d$. For $p \in (0,\infty)$, set
\begin{align*}
	M_{p}(x,y)\coloneq\left( \frac1d \sum_{i=1}^d c_{i}^{\,p} \right)^{1/p},
	\quad
	M_{0}(x,y)\coloneq\Bigl( \prod_{i=1}^d c_{i} \Bigr)^{1/d},\quad
	M_{\infty}(x,y)\coloneq\max_{1\le i\le d} c_{i}.
\end{align*} 
Then the trace, norm, and house Euclidean properties can be viewed as
special cases of the $p$-norm Euclidean property.
\begin{defn}\label{defn:p-norm}
	Let $p\in [0,\infty]$. We say that $L$ is $p$-norm Euclidean if for every $x\in V$, there exists $y\in L$ such that $M_p(x,y) < 1$.
\end{defn}
By the power mean inequality, $M_p(x,y)\le M_q(x,y)$ for $p<q$. Thus, for $p\in [0,\infty]$, we have the following relation:
\begin{align}\label{p-norm-rela}
	\text{$L$ is house Euclidean ($p=\infty$)}\; \Rightarrow\; \text{$L$ is $p$-norm Euclidean}\;  \Rightarrow\; \text{$L$ is norm Euclidean ($p=0$)}.
\end{align}
Note that $M_{1}(x,y)=\tr_{F/\mathbb{Q}}(Q(x-y))/d$. Hence $L$ is $1$-norm Euclidean if and only if it is trace Euclidean, which in turn implies that $L$ is norm Euclidean, by the inequality
\begin{align*}
	\nr_{F/\mathbb{Q}}(Q(x-y))
	\le
	\left(
	\frac{\tr_{F/\mathbb{Q}}(Q(x-y))}{d}
	\right)^d
	<1.
\end{align*}
Every norm Euclidean lattice is $\mathcal{O}_{F}$-maximal.
Indeed, if $L\subsetneq L'$ with $L'$ integral, choose $x\in L'\setminus L$. For every $y\in L$, since $L'$ is integral we have $Q(x-y)\in\mathcal{O}_{F}$, and positive definiteness implies $Q(x-y)\neq 0$, so
\begin{align*}
	\nr_{F/\mathbb{Q}}(Q(x-y))\ge 1,
\end{align*} 
a contradiction (see also \cite[Proposition~11]{clark_ADC-I-2012}). Thus, for $p\in[1,\infty]$, we further have 
\begin{align}\label{implication-Euclidean-maximal}
	\begin{array}{ccccc}
		\text{$L$ is house Euclidean} & \Longrightarrow & \text{$L$ is $p$-norm Euclidean} & \Longrightarrow & \text{$L$ is $1$-norm Euclidean} \\
		& & & & \big\Updownarrow \\
		\text{$L$ is $\mathcal{O}_{F}$-maximal} & \Longleftarrow & \text{$L$ is norm Euclidean} & \Longleftarrow & \text{$L$ is trace Euclidean}
	\end{array}
\end{align}

Combining these implications with Theorem \ref{thm:finiteness-integral}, we have the following corollary.
\begin{cor}\label{cor:p-norm}
	Let $p\in[1,\infty]$. As $F$ varies over all totally real number fields, there are only finitely many positive definite integral $p$-norm Euclidean $\mathcal{O}_{F}$-lattices of arbitrary rank $n\ge1$, up to isometry. In particular, this holds for house Euclidean lattices.
\end{cor}
\begin{proof}
	For $p\ge1$, the power mean inequality gives $M_{1}(x,y)\le M_{p}(x,y)$. Thus every $p$-norm Euclidean lattice is trace Euclidean, and Theorem \ref{thm:finiteness-integral} applies.
\end{proof}
The following theorem classifies all trace Euclidean lattices of rank one over real quadratic fields.
\begin{thm}\label{thm:rank-one-classification}
	Let $F=\mathbb{Q}(\sqrt m)$ be a real quadratic field with square-free $m>1$,
	and let $L$ be a positive definite integral $\mathcal{O}_{F}$-lattice of rank one. Then $L$ is trace Euclidean if and only if
	$L\cong\langle a\rangle$ over $F$ for one of the following pairs:
	\begin{align*}
		(m,a)\in \{(2,1),(3,2+\sqrt{3}),(5,1),(5,2),(13,1),(21,(5+\sqrt{21})/2)\}.
	\end{align*}
	In particular, every such lattice is free, and there are exactly six
	isometry classes as $F$ varies over all real quadratic fields.
\end{thm}
Recall that a totally real number field $F$ is \emph{Euclidean} if its ring of integers $\mathcal{O}_{F}$ is a Euclidean domain; that is, there exists a function $\psi\colon \mathcal{O}_{F}\setminus\{0\}\to\mathbb{N}$ such that, for any $a,b\in\mathcal{O}_{F}$ with $b\neq 0$, there exist $q,r\in\mathcal{O}_{F}$ satisfying $a=qb+r$ and either $r=0$ or $\psi(r)<\psi(b)$.

If one can take
$\psi(\cdot)=|\nr_{F/\mathbb{Q}}(\cdot)|$,
then $F$ is called \textit{norm Euclidean}.
Let $m>1$ be a square-free integer. It is known from \cite[Theorem 4.4]{lemmermeyer_euclidean_1995} that a real quadratic field $\mathbb{Q}(\sqrt{m})$ is norm Euclidean if and only if
\begin{align*}
	m\in \{2,3,5,6,7,11,13,17,19,21,29,33,37,41,57,73\}.
\end{align*}
We introduce the following definition as the trace analogue of norm Euclidean fields.
\begin{defn}
	\label{defn:trace-euclidean-field}
	Let $F$ be a totally real number field of degree $d$. 
	We say that $F$ is trace Euclidean if for every $x\in F$, 
	there exists $y\in\mathcal{O}_F$ such that
	\begin{align*}
		\tr_{F/\mathbb{Q}}\bigl((x-y)^2\bigr) < d.
	\end{align*}
\end{defn}
By Theorem \ref{thm:rank-one-classification} with $a=1$, i.e., $Q(x)=x^{2}$, we have the following corollary.
\begin{cor}\label{cor:trace-euclidean-quadratic}
	Let $F=\mathbb{Q}(\sqrt{m})$ be a real quadratic field with square-free $m>1$.
	Then $F$ is trace Euclidean if and only if $m\in\{2,5,13\}$.
\end{cor}
Take the quadratic space $V=F$ with quadratic form $Q(x)=x^2$ and lattice $L=\mathcal{O}_F$. Then 
\begin{align*}
	&\text{$F$ is a trace Euclidean (resp. norm Euclidean) field}\\
	\Longleftrightarrow \quad&\text{$L$ is a trace Euclidean (resp. norm Euclidean)  $\mathcal{O}_F$-lattice}.
\end{align*}
In particular, Theorem \ref{thm:finiteness-classic} with $n=1$ shows that there are only finitely many totally real trace Euclidean fields, up to isomorphism. Combining \eqref{implication-Euclidean-maximal} with the fact that every Euclidean domain is a PID, we obtain
\begin{align*}
	\text{$F$ is trace Euclidean} \implies \text{$F$ is norm Euclidean} \implies \text{$F$ is Euclidean} \implies \text{$\mathcal{O}_{F}$ is PID}.
\end{align*}
However, the converse of the first implication does not hold in general. For example, the quadratic field $\mathbb{Q}(\sqrt{3})$ is norm Euclidean, but is not trace Euclidean, by Corollary \ref{cor:trace-euclidean-quadratic}.   

The paper is organized as follows.
In Section \ref{sec:preliminaries}, we study the analytic properties of the function $\Phi_{\mathbb{R}}$ and give an equivalent condition  for trace Euclidean lattices.
In Section \ref{sec:volume-trace}, we establish the relation between the covolume and the trace covering radius, and then derive an upper bound for the norm of the volume ideal of a trace Euclidean lattice.
In Section \ref{sec:root-discriminant-bounds}, we derive root discriminant bounds and obtain effective estimates. In Section \ref{sec:finiteness}, we prove the finiteness theorems. In Section \ref{sec:trace-euclidean-fields}, we classify all positive definite integral trace Euclidean lattices of rank one over real quadratic fields and determine all trace Euclidean real quadratic fields.

\section{Preliminaries}
\label{sec:preliminaries}
In this section, we assume that $F$ is totally real of degree $d$. Let $V$ be an $n$-dimensional quadratic space over $F$ associated with a quadratic form $Q$, and let $L$ be a positive definite $\mathcal{O}_{F}$-lattice on $V$.
For each real embedding $\sigma_i \colon F \to \mathbb{R}$, let $V_i = V \otimes_{F,\sigma_i} \mathbb{R}$ with induced quadratic form $Q_i \colon V_i \to \mathbb{R}$ given by $Q_i(x \otimes 1) = \sigma_i(Q(x))$. We also use $\sigma_{i}$ to denote the embedding $\sigma_i \colon V \to V_i$, $\sigma_{i}(x)=x\otimes 1$. Write $V_{\mathbb{R}}=V\otimes_{\mathbb{Q}}\mathbb{R}$ and
let $\sigma_{V}\colon V\to V_{\mathbb{R}}$ be the natural embedding 
$\sigma_{V}(x)=x\otimes 1$.

Choose a $\mathbb{Z}$-basis $\{v_{1},\ldots,v_{nd}\}$ of $L$. For
$u=(u_{1},\ldots,u_{nd})^{\mathsf{T}}\in\mathbb{R}^{nd}$, define
$\varphi\colon\mathbb{R}^{nd}\to V_{\mathbb{R}}$ by
\begin{align*}
	\varphi(u)\coloneq\sum_{i=1}^{nd}u_{i}\sigma_{V}(v_{i}).
\end{align*}
Then we have the following $\mathbb{R}$-linear isomorphisms
\begin{align}\label{Riso}
	\mathbb{R}^{nd}\xrightarrow{\varphi}V_{\mathbb{R}}
	\cong\prod_{i=1}^{d}V_{i}.
\end{align} 
Indeed, since $F$ is totally real, the map
$\alpha\otimes r\mapsto(r\sigma_{1}(\alpha),\ldots,r\sigma_{d}(\alpha))$
gives an $\mathbb{R}$-linear isomorphism
$F\otimes_{\mathbb{Q}}\mathbb{R}\cong\mathbb{R}^{d}$. Hence
\begin{align}\label{VR}
	V_{\mathbb{R}}
	&\cong V\otimes_{F}(F\otimes_{\mathbb{Q}}\mathbb{R}) \cong\prod_{i=1}^{d}(V\otimes_{F,\sigma_{i}}\mathbb{R})
	=\prod_{i=1}^{d}V_{i}.
\end{align}
Since $\mathbb{Q}\mathcal{O}_{F}=F$ and $FL=V$, we have
\begin{align*}
	V=FL=(\mathbb{Q}\mathcal{O}_{F})L=\mathbb{Q}L.
\end{align*}
Thus $\{v_{1},\ldots,v_{nd}\}$ is a $\mathbb{Q}$-basis of $V$, and
$\{\sigma_{V}(v_{1}),\ldots,\sigma_{V}(v_{nd})\}$ is an
$\mathbb{R}$-basis of $V_{\mathbb{R}}$. This proves that $\varphi$ is an
$\mathbb{R}$-linear isomorphism.

Under the isomorphism \eqref{VR}, $\sigma_{V}(x)=
(\sigma_{1}(x),\ldots,\sigma_{d}(x))$ for $x\in V$.
Define the maps
\begin{align*}
	\sigma_{F} &\colon F \longrightarrow \mathbb{R}^{d}, \quad \alpha \mapsto (\sigma_{1}(\alpha), \dots, \sigma_{d}(\alpha)), \\
	Q_{\mathbb{R}} &\colon V_{\mathbb{R}}\to \mathbb{R}^{d},\quad x\mapsto (Q_{1}(x_{1}),\ldots,Q_{d}(x_{d})),\quad x_{i}\in V_{i}.
\end{align*}
Then the following diagram commutes,
\begin{equation}\label{com-diag}
	\begin{CD}
		V @>Q>> F \\
		@V\sigma_{V} VV @VV\sigma_{F} V \\
		V_{\mathbb{R}} @>>Q_{\mathbb{R}}> \mathbb{R}^{d}
	\end{CD}
\end{equation}
When no confusion can arise, we omit the subscripts and write $\sigma$
for both $\sigma_F$ and $\sigma_V$. The image $\Lambda=\sigma(L)$ is a full rank
$\mathbb{Z}$-lattice in $V_{\mathbb{R}}$. Hence
$\sigma(V)=\mathbb{Q}\Lambda$ is dense in $V_{\mathbb{R}}$.

Write $\phi=\tr_{F/\mathbb{Q}}$ for short and define the function $\phi_{\mathbb{R}}\colon\mathbb{R}^d\to\mathbb{R}$ by $
	\phi_{\mathbb{R}}(\alpha_{1},\dots,\alpha_{d}) = \sum_{i=1}^d \alpha_{i}$.
From the commutative diagram
\eqref{com-diag}, we have
\begin{align}\label{phiQ}
	\phi(Q(x))
	=\phi_{\mathbb{R}}(\sigma_F(Q(x)))
	=\phi_{\mathbb{R}}(Q_{\mathbb{R}}(\sigma_V(x)))
\end{align}
for $x\in V$. Since each $Q_i$ is positive definite, $\phi_{\mathbb{R}}\circ Q_{\mathbb{R}}$ is nonnegative on
$V_{\mathbb{R}}$. 

We define the functions
\begin{align*}
	\Phi(x) &\coloneq \inf_{y\in L} \phi(Q(x-y)), \quad x\in V, \\
	\Phi_{\mathbb{R}}(z) &\coloneq\inf_{\lambda\in \sigma(L)} \phi_{\mathbb{R}}(Q_{\mathbb{R}}(z-\lambda)), \quad z\in V_{\mathbb{R}},
\end{align*}
and define the \textit{trace covering radius} of $L$ by
\begin{align*}
	\rho_T(L)
	\coloneq
	\sup_{z\in V_{\mathbb{R}}}\Phi_{\mathbb{R}}(z)
	=
	\sup_{z\in V_{\mathbb{R}}}
	\inf_{\lambda\in\sigma(L)}
	\phi_{\mathbb{R}}
	\left(Q_{\mathbb{R}}(z-\lambda)\right).
\end{align*}
Let $B_{i}$ be the symmetric bilinear form
associated with $Q_{i}$. For $z=(z_{1},\ldots,z_{d})^{\mathsf{T}},\lambda=(\lambda_{1},\ldots,\lambda_{d})^{\mathsf{T}}\in V_{\mathbb{R}}$, where $z_{i},\lambda_{i}\in V_{i}$, we define
\begin{align*}
	q_{T}(z)
	&\coloneq
	\phi_{\mathbb{R}}\bigl(Q_{\mathbb{R}}(z)\bigr)
	=
	\sum_{i=1}^{d}Q_{i}(z_{i})\quad\text{and}\quad
	B_{T}(z,\lambda)
	\coloneq
	\sum_{i=1}^{d}B_{i}(z_{i},\lambda_{i}),
\end{align*}
and $
\|z\|_{T}
\coloneq \sqrt{q_{T}(z)}$. From the commutative diagram
\eqref{com-diag} again, we have
\begin{align}\label{eq:trace-form-compatibility}
	q_{T}(\sigma(x))
	=
	\tr_{F/\mathbb{Q}}(Q(x)) \quad \text{and}\quad
	B_{T}(\sigma(x),\sigma(y))
	&=
	\tr_{F/\mathbb{Q}}(B(x,y))
\end{align}
for $x,y\in V$. In particular,
\begin{align}\label{rhoTL=supinf}
	\rho_T(L)
	=
	\left(
	\sup_{z\in V_{\mathbb{R}}}
	\inf_{\lambda\in\Lambda}
	\|z-\lambda\|_T
	\right)^2.
\end{align}
Thus $\rho_{T}(L)$ is the square of the covering radius with respect to $\|\cdot\|_{T}$.

\begin{prop}\label{prop:trace-coordinates}
	 The following statements hold.
	\begin{enumerate}[label=\rm(\roman*), itemindent=0em, leftmargin=*]
		\item $
			\varphi(\mathbb{Z}^{nd})=\Lambda$ and $ \varphi(\mathbb{Q}^{nd})=\sigma(V)$.
			
		\item The quadratic form $q(u)\coloneq q_{T}(\varphi(u))$,
		$u\in\mathbb{R}^{nd}$, is positive definite and integral.
		\item For every $u\in\mathbb{R}^{nd}$, we have
		\begin{align*}
			\Phi_{\mathbb{R}}(\varphi(u))
			=\inf_{v\in\mathbb{Z}^{nd}}q(u-v).
		\end{align*}
	\end{enumerate}
\end{prop}
\begin{proof}
	(i) Since $\{v_{1},\ldots,v_{nd}\}$ is a $\mathbb{Z}$-basis of $L$
	and $V=\mathbb{Q}L$, by the definition of $\varphi$ and the $\mathbb{Q}$-linearity of $\sigma$, we have
	\begin{align*}
		\varphi(\mathbb{Z}^{nd})
		&=\sum_{i=1}^{nd}\mathbb{Z}\sigma(v_{i})
		=\sigma(L)=\Lambda\quad\text{and}\quad
		\varphi(\mathbb{Q}^{nd}) =\sum_{i=1}^{nd}\mathbb{Q}\sigma(v_{i})
		=\sigma(V).
	\end{align*}
	
	(ii) Since $q_{T}$ is positive definite and $\varphi$ is an isomorphism,
	$q$ is positive definite. For $m=(m_{1},\ldots,m_{nd})^{\mathsf{T}}
	\in\mathbb{Z}^{nd}$, we have $\sum_{i=1}^{nd}m_{i}v_{i}\in L$.
	Since $L$ is integral, $Q(L)\subseteq\mathcal{O}_{F}$.
	By \eqref{eq:trace-form-compatibility},
	\begin{align*}
		q(m)=\tr_{F/\mathbb{Q}}
		\left(Q\left(\sum_{i=1}^{nd}m_{i}v_{i}\right)\right)
		\in\tr_{F/\mathbb{Q}}(\mathcal{O}_{F})\subseteq\mathbb{Z}.
	\end{align*}
	Thus $q(\mathbb{Z}^{nd})\subseteq\mathbb{Z}$, so $q$ is integral.
	
	(iii) By (i), the linearity of $\varphi$, and the definition of $q$, we see that
	\begin{align*}
		\Phi_{\mathbb{R}}(\varphi(u))
		&=\inf_{\lambda\in\Lambda}q_{T}(\varphi(u)-\lambda)=\inf_{v\in\mathbb{Z}^{nd}}
		q_{T}(\varphi(u-v))
		=\inf_{v\in\mathbb{Z}^{nd}}q(u-v).
	\end{align*}
\end{proof}

The next lemma collects some analytic properties of the function $\Phi_{\mathbb{R}}$.

\begin{lem}\label{lem:Phi}
	The following statements hold.
	\begin{enumerate}[label=\rm(\roman*), itemindent=0em, leftmargin=*]
		\item $\Phi_{\mathbb{R}}(z+\lambda) = \Phi_{\mathbb{R}}(z)$ for $z \in V_{\mathbb{R}}$ and $\lambda \in \sigma(L)$;
		\item $\Phi_{\mathbb{R}}$ is continuous on $V_{\mathbb{R}}$;
		\item $\sup_{z \in V_{\mathbb{R}}} \Phi_{\mathbb{R}}(z)=\Phi_{\mathbb{R}}(z_{0})$ for some $z_0 \in V_{\mathbb{R}}$;
		\item $\sup_{x \in V} \Phi(x) = \sup_{z \in V_{\mathbb{R}}} \Phi_{\mathbb{R}}(z)$;
		\item $\max_{x \in V} \Phi(x) = \max_{z \in V_{\mathbb{R}}} \Phi_{\mathbb{R}}(z)$.
	\end{enumerate}
\end{lem}

\begin{proof}
	(i) For any $\lambda \in \Lambda$, we have
	\begin{align*}
		\Phi_{\mathbb{R}}(z+\lambda) = \inf_{w\in\Lambda} \phi_{\mathbb{R}}\bigl(Q_{\mathbb{R}}(z-w+\lambda)\bigr)
		= \inf_{\lambda^{\prime}\in\Lambda} \phi_{\mathbb{R}}\bigl(Q_{\mathbb{R}}(z-\lambda^{\prime})\bigr) = \Phi_{\mathbb{R}}(z),
	\end{align*}
	since $\Lambda$ is an additive subgroup of $V_{\mathbb{R}}$, and $\lambda^{\prime} = w- \lambda$ ranges over $\Lambda$ as $w$ does.
	
	(ii) Recall that
	\begin{align*}
		\phi_{\mathbb R}\bigl(Q_{\mathbb R}(z)\bigr)=
	q_T(z)
		=
		\|z\|_T^2
	\end{align*} 
	for every \(z\in V_{\mathbb R}\). Fix $z_{0} \in V_{\mathbb{R}}$ and $\varepsilon > 0$. By definition of the infimum, there exists $\lambda_0 \in \Lambda$ such that
	\begin{align}\label{eq:inf-approx}
		\phi_{\mathbb{R}}\bigl(Q_{\mathbb{R}}(z_{0}-\lambda_{0})\bigr) < \Phi_{\mathbb{R}}(z_{0}) + \frac{\varepsilon}{2}.
	\end{align}
	Since both $Q_{\mathbb{R}}$ and $\phi_{\mathbb{R}}$ are continuous, the map $
		z\mapsto
		\phi_{\mathbb{R}}\left(Q_{\mathbb{R}}(z-\lambda_{0})\right)$
	is continuous. Hence there exists $\delta>0$ such that, for all
	$z\in V_{\mathbb{R}}$ with $\|z-z_{0}\|_T<\delta$,
	\begin{align}\label{eq:cont-est}
		\phi_{\mathbb{R}}\left(Q_{\mathbb{R}}(z-\lambda_{0})\right) < \phi_{\mathbb{R}}\left(Q_{\mathbb{R}}(z_{0}-\lambda_{0})\right) + \frac{\varepsilon}{2}.
	\end{align}
	Combining \eqref{eq:inf-approx} with \eqref{eq:cont-est}, we obtain
	\begin{align*}
		\Phi_{\mathbb{R}}(z) \le \phi_{\mathbb{R}}\left(Q_{\mathbb{R}}(z-\lambda_{0})\right) < \Phi_{\mathbb{R}}(z_{0})+\varepsilon,
	\end{align*} 
	so $\Phi_{\mathbb{R}}$ is upper semicontinuous at $z_{0}$.
		Next, we prove lower semicontinuity. If
	$\varepsilon\geq\Phi_{\mathbb{R}}(z_{0})$, then $
		\Phi_{\mathbb{R}}(z)
		\geq
		0
		\geq
		\Phi_{\mathbb{R}}(z_{0})-\varepsilon$,
	so there is nothing to prove. Thus we may assume that $
		0<\varepsilon<\Phi_{\mathbb{R}}(z_{0})$. 
	Choose $r>1+\|z_0\|_T$ such that
	\begin{align*}
		(r-1)^2
		\geq
		\Phi_{\mathbb{R}}(z_{0})-\varepsilon.
	\end{align*}
	If $\|z_{0}-\lambda\|_T\geq r$ and $\|z-z_{0}\|_T<1$, then
	\begin{align*}
		\|z-\lambda\|_T
		\geq
		\|z_{0}-\lambda\|_T-\|z-z_{0}\|_T
		>
		r-1,
	\end{align*}
	and hence
	\begin{align*}
		\phi_{\mathbb{R}}
		\left(Q_{\mathbb{R}}(z-\lambda)\right)
		=
		q_T(z-\lambda)
		=
		\|z-\lambda\|_T^2
		>
		(r-1)^2
		\geq
		\Phi_{\mathbb{R}}(z_{0})-\varepsilon.
	\end{align*}
	Since $\Lambda$ is discrete, there are only finitely many $\lambda\in\Lambda$ with $\|z_{0}-\lambda\|_{T}<r$, say $\lambda_{1},\ldots,\lambda_{k} \in \Lambda$. For each $\lambda_{i}$, the function $z\mapsto\phi_{\mathbb{R}}(Q_{\mathbb{R}}(z-\lambda_{i}))$ is continuous. So there exists $\delta_{i}>0$ such that whenever $\|z-z_{0}\|_{T}<\delta_{i}$, 
	\begin{align*}
		\phi_{\mathbb{R}}(Q_{\mathbb{R}}(z-\lambda_{i})) > \phi_{\mathbb{R}}(Q_{\mathbb{R}}(z_{0}-\lambda_{i})) - \varepsilon \ge \Phi_{\mathbb{R}}(z_{0}) - \varepsilon.
	\end{align*} 
	Set $\delta=\min\{1,\delta_{1},\ldots,\delta_{k}\}>0$. Then, whenever $\|z-z_{0}\|_{T}<\delta$,
	\begin{align*}
		\phi_{\mathbb{R}}(Q_{\mathbb{R}}(z-\lambda))\ge \Phi_{\mathbb{R}}(z_{0}) - \varepsilon 
	\end{align*}
	for all  $\lambda \in \Lambda$. Hence $\Phi_{\mathbb{R}}(z) \ge \Phi_{\mathbb{R}}(z_{0}) - \varepsilon$.
	
	By the above discussion, we see that $\Phi_{\mathbb{R}}$ is continuous on $V_{\mathbb{R}}$.
	
	(iii) Since $\Lambda$ is a full rank lattice in $V_{\mathbb{R}}$, the quotient space $V_{\mathbb{R}}/\Lambda$ is compact. Define
	\begin{align*}
		 \widetilde{\Phi}_{\mathbb{R}}:V_{\mathbb{R}}/\Lambda \to \mathbb{R}_{\ge 0}, \quad \widetilde{\Phi}_{\mathbb{R}}(\bar{z}) = \Phi_{\mathbb{R}}(z),
	\end{align*}
	which is well-defined by periodicity (property (i)). By (ii), $\Phi_{\mathbb{R}}$ is continuous. Since $p\colon V_{\mathbb{R}}\to V_{\mathbb{R}}/\Lambda$ is a quotient map and  $
	\Phi_{\mathbb{R}}
	=
	\widetilde{\Phi}_{\mathbb{R}}\circ p$, it follows that $\widetilde{\Phi}_{\mathbb{R}}$ is continuous. Hence, by the extreme value theorem, there exists some $\bar{z}_{0}\in V_{\mathbb{R}}/\Lambda $ such that
	\begin{align*}
		\widetilde{\Phi}_{\mathbb{R}}(\bar{z}_{0}) = \max_{\bar{z} \in V_{\mathbb{R}}/\Lambda } \widetilde{\Phi}_{\mathbb{R}}(\bar{z}) 
	\end{align*}
	Choose $z_{0}\in V_{\mathbb{R}}$ such that $z_{0}+\Lambda=\bar{z}_{0}$. Then, for any $z\in V_{\mathbb{R}}$,
	\begin{align*}
		\Phi_{\mathbb{R}}(z)=	\widetilde{\Phi}_{\mathbb{R}}(\bar{z})\le  	\widetilde{\Phi}_{\mathbb{R}}(\bar{z}_{0})=\Phi_{\mathbb{R}}(z_{0}).
	\end{align*}
	Clearly, $\Phi_{\mathbb{R}}(z_{0})\le \sup_{z \in V_{\mathbb{R}}} \Phi_{\mathbb{R}}(z) $ and hence $\sup_{z \in V_{\mathbb{R}}} \Phi_{\mathbb{R}}(z)=\Phi_{\mathbb{R}}(z_{0})$.
	
	(iv) For $x\in V$ and $y\in L$, from the definition of
	$\phi_{\mathbb{R}}$ and \eqref{phiQ}, we have
	\begin{align*}
		\phi(Q(x-y))
		&=\phi_{\mathbb{R}}(\sigma(Q(x-y)))=\phi_{\mathbb{R}}
		\bigl(Q_{\mathbb{R}}(\sigma(x)-\sigma(y))\bigr).
	\end{align*}
	Taking the infimum over $y\in L$, and noting that
	$\sigma(L)=\Lambda$, we obtain
	\begin{align}\label{eq:Phi-compatibility}
		\Phi(x)
		&=\inf_{y\in L}
		\phi_{\mathbb{R}}
		\bigl(Q_{\mathbb{R}}(\sigma(x)-\sigma(y))\bigr)=\inf_{\lambda\in\Lambda}
		\phi_{\mathbb{R}}
		\bigl(Q_{\mathbb{R}}(\sigma(x)-\lambda)\bigr)
		=\Phi_{\mathbb{R}}(\sigma(x)).
	\end{align}
	Therefore,
	\begin{align*}
		\sup_{x\in V}\Phi(x)=\sup_{z\in \sigma(V)}\Phi_{\mathbb{R}}(z)
		\leq
		\sup_{z\in V_{\mathbb{R}}}\Phi_{\mathbb{R}}(z).
	\end{align*}
	Conversely, let $z_{0}\in V_{\mathbb{R}}$. Since $\sigma(V)$ is dense in
	$V_{\mathbb{R}}$, there exists a sequence $\{v_n\}\subset V$ such that
	$\lim_{n\to\infty}\sigma(v_n)=z_{0}$.
	By (ii) and \eqref{eq:Phi-compatibility},
	\begin{align*}
		\Phi_{\mathbb{R}}(z_{0})=\lim_{n\to\infty}\Phi_{\mathbb{R}}(\sigma(v_n))=\lim_{n\to\infty}\Phi(v_n)
		\leq\sup_{x\in V}\Phi(x).
	\end{align*}
	Since $z_{0}$ is arbitrary,
	\begin{align*}
		\sup_{z\in V_{\mathbb{R}}}\Phi_{\mathbb{R}}(z)
		\leq\sup_{x\in V}\Phi(x).
	\end{align*}
	Combining the two inequalities, we are done.
	
	(v) By Proposition \ref{prop:trace-coordinates}(ii) and (iii), $q$ is a positive definite integral quadratic form and
	\begin{align*}
		\Phi_{\mathbb{R}}(\varphi(u))
		=\inf_{v\in\mathbb{Z}^{nd}}q(u-v).
	\end{align*}
	
	Hence, by \eqref{Riso} and
	\cite[Proposition 4.1(a)]{clark_ADC-II-2014}, there exists
	$u_{0}\in\mathbb{Q}^{nd}$ such that
	\begin{align}\label{maxphi}
		\max_{z\in V_{\mathbb{R}}}\Phi_{\mathbb{R}}(z)
		&=\max_{u\in\mathbb{R}^{nd}}
		\Phi_{\mathbb{R}}(\varphi(u))
		=\Phi_{\mathbb{R}}(\varphi(u_{0})).
	\end{align}
	Since $\varphi(\mathbb{Q}^{nd})=\sigma(V)$ by Proposition \ref{prop:trace-coordinates}(i), we may write
	$\varphi(u_{0})=\sigma(x_{0})$ for some $x_{0}\in V$.
	By \eqref{eq:Phi-compatibility}, for every $x\in V$, we have
	\begin{align}\label{phix-x0}
		\Phi(x)=\Phi_{\mathbb{R}}(\sigma(x))
		\leq\Phi_{\mathbb{R}}(\varphi(u_{0}))=\Phi_{\mathbb{R}}(\sigma(x_{0}))=\Phi(x_{0}).
	\end{align}
 	Combining \eqref{phix-x0} with \eqref{maxphi}, we deduce that
	\begin{align*}
		\max_{x\in V}\Phi(x)=\Phi(x_{0})=\Phi_{\mathbb{R}}(\varphi(u_{0})) =\max_{z\in V_{\mathbb{R}}}\Phi_{\mathbb{R}}(z).
	\end{align*}
\end{proof}
\begin{lem}\label{lem:phi-Euclidean-trace}
	Let $L$ be an $\mathcal{O}_F$-lattice of rank $n$ and put $d=[F:\mathbb{Q}]$.
	Then $L$ is trace Euclidean if and only if $\rho_T(L)<d$.
\end{lem}

\begin{proof}
	By Lemma \ref{lem:Phi}\textup{(v)}, there exists $x_{0}\in V$ such that
	\begin{align*}
		\rho_T(L)=\max_{x\in V}\Phi(x)=\Phi(x_{0}).
	\end{align*}
	If $L$ is trace Euclidean, then there exists $y_{0}\in L$ such that
	\begin{align*}
		\rho_T(L)=\Phi(x_{0})
		\leq\tr_{F/\mathbb{Q}}\bigl(Q(x_{0}-y_{0})\bigr)<d.
	\end{align*}
	Conversely, suppose that $\rho_T(L)<d$. For every $x\in V$,
	\begin{align*}
		\inf_{y\in L}\tr_{F/\mathbb{Q}}\bigl(Q(x-y)\bigr)
		=\Phi(x)\leq\rho_T(L)<d.
	\end{align*}
	By the definition of the infimum, there exists $y\in L$ such that
	$\tr_{F/\mathbb{Q}}(Q(x-y))<d$. Thus $L$ is trace Euclidean.
\end{proof}

\section{Covolumes and the trace covering radius of lattices}
\label{sec:volume-trace}

In this section, we retain the notation of
Section \ref{sec:preliminaries}. Let $L$ be a positive definite integral
$\mathcal{O}_{F}$-lattice of rank $n\ge 1$ on $V$, and put $ 	\Lambda=\sigma_{V}(L)\subseteq V_{\mathbb{R}}\cong \mathbb{R}^{nd}$. Let $z_{1},\ldots,z_{nd}$ be a $\mathbb{Z}$-basis of
$\Lambda$, and let $
	G_{\Lambda}
	= (B_{T}(z_{i},z_{j}))_
	{1\leq i,j\leq nd}$
be the Gram matrix of $\Lambda$ with respect to $B_{T}$.

Let $\vol$ be the Riemannian volume measure on
$V_{\mathbb{R}}$ induced by $B_{T}$, normalized so that the
unit cube spanned by any orthonormal basis
(with respect to $B_{T}$) has volume $1$. The set
\begin{align*}
	\mathcal{F}
	=
	\left\{
	\sum_{i=1}^{nd}c_{i}z_{i}\mid
	0\leq c_{i}<1
	\right\}
\end{align*}
is called a \textit{fundamental domain} or \textit{fundamental mesh} for $\Lambda$. From
\cite[Chapter~I, \S4, Definition~(4.1) and p.~26]
{neukirch_algebraic_1999}, we have $\vol(\mathcal{F})=\sqrt{\det G_{\Lambda}}$, and write
\begin{align}\label{vol-defn}
	\vol(\Lambda)
	\coloneq
	\vol(\mathcal{F})
	=
	\sqrt{\det G_{\Lambda}},
\end{align}
which is also called the \textit{covolume} of $\Lambda$; see \cite[p.\hskip 0.1cm 392]{hoffstein_introduction_2014}. Following \cite[Definition~2.2]{cesaroni-novaga_minimal_2024}, we define
the \textit{Voronoi cell} of $\Lambda$ with respect to $\|\cdot\|_{T}$ by
\begin{align*}
	\mathcal{V}
	=
	\left\{
	z\in V_{\mathbb{R}}:
	\|z\|_{T}\leq \|z-\lambda\|_{T}
	\text{ for every }\lambda\in\Lambda
	\right\}.
\end{align*}
By
\cite[Remark 2.8]{cesaroni-novaga_minimal_2024},
$\mathcal{V}$ is a fundamental domain for $\Lambda$ up to its
boundary. Hence
\begin{align}\label{eq:voronoi-covolume}
	\vol(\mathcal{V})
	=
	\vol(\mathcal{F})
	=
	\vol(\Lambda).
\end{align}

We now compute the covolume of $\Lambda$ in terms of the field discriminant of $F$ and the volume ideal of $L$.
\begin{lem}\label{lem:gram-det}
	Let $L$ be an $\mathcal{O}_{F}$-lattice of rank $n$ and $G_{\Lambda}$ be the Gram matrix of $\Lambda=\sigma_{V}(L)$ with respect to the trace inner product $B_{T}$. Then
	\begin{align*}
		\det G_{\Lambda}
		=
		\Delta_{F}^{n}\nr(\mathfrak{v}(L)).
	\end{align*}
\end{lem}

\begin{proof}
	By \cite[Theorem 9.3.6]{voight_quaternion_2021}, there exists a pseudobasis
	\begin{align*}
		L
		=
		\mathfrak{a}_{1}x_{1}
		\oplus\cdots\oplus
		\mathfrak{a}_{n}x_{n},
	\end{align*}
	where $\{x_{1},\ldots,x_{n}\}$ is a basis of $V$ over $F$ and each $\mathfrak{a}_{i}$ is a nonzero fractional ideal of $F$. For each $i$, choose a $\mathbb{Z}$-basis
	$\{\alpha_{i\ell}\}_{1\le\ell\le d}$ of $\mathfrak{a}_{i}$. Then $\Lambda$ has the $\mathbb{Z}$-basis
	\begin{align*}
		 \{
		\sigma_{V}(\alpha_{i\ell}x_{i})
		\mid
		1\le i\le n,\ 1\le\ell\le d \},
	\end{align*}
	ordered first by $i$ and then by $\ell$. Since a change of $\mathbb{Z}$-basis preserves the Gram determinant,
	we may take this as the basis defining $G_{\Lambda}$. Put $
		G_{L}
		=
		(b_{ij})_{1\le i,j\le n}$ and $
		b_{ij}
		=
		B(x_{i},x_{j})$, where $B$ denotes the bilinear form of $L$. For $\alpha,\beta\in F$, we have
	\begin{equation}\label{eq:inner-product-sigma}
		\begin{aligned}
			B_{T}\bigl(\sigma_{V}(\alpha x_{i}),
			\sigma_{V}(\beta x_{j})\bigr)
			&=
			\sum_{k=1}^{d}\sigma_{k}(\alpha b_{ij}\beta)=
			\sum_{k=1}^{d}
			\sigma_{k}(\alpha)\sigma_{k}(b_{ij})
			\sigma_{k}(\beta).
		\end{aligned}
	\end{equation}
	For $1\le i\le n$, define $
		\Omega_{i}
		=
		\left(\sigma_{k}(\alpha_{i\ell})\right)_
		{1\le k,\ell\le d}$,
	and, for $1\le i,j\le n$ and $1\le k\le d$, put
	\begin{align*}
		D_{ij}^{(k)}
		=
		\diag\left(
		\sigma_{k}(b_{ij}),
		\sigma_{k+1}(b_{ij}),
		\ldots,
		\sigma_{d}(b_{ij})
		\right),
		\qquad
		D_{ij}=D_{ij}^{(1)}.
	\end{align*}
    Thus $D_{ij}^{(k)}$ has order $d-k+1$. Regard $G=G_{\Lambda}$ as an $n\times n$ block matrix with blocks of order $d$. It follows from \eqref{eq:inner-product-sigma} that
	\begin{align*}
		G_{ij}
		=
		\Omega_{i}^{T}D_{ij}\Omega_{j}.
	\end{align*}
	Define $C=
	\diag(\Omega_{1},\ldots,\Omega_{n})$ and
	\begin{align*}
		A^{(k)}
		&=
		\begin{pmatrix}
			D_{11}^{(k)} & D_{12}^{(k)} & \cdots & D_{1n}^{(k)}\\
			D_{21}^{(k)} & D_{22}^{(k)} & \cdots & D_{2n}^{(k)}\\
			\vdots       & \vdots       &        & \vdots\\
			D_{n1}^{(k)} & D_{n2}^{(k)} & \cdots & D_{nn}^{(k)}
		\end{pmatrix}.
	\end{align*}
	Then $A^{(k)}$ has order $n(d-k+1)$ and $
		G
		=
		C^{T}A^{(1)}C$. So
	\begin{align}\label{detG}
		\det G
		=
		(\det C)^{2}\det A^{(1)}.
	\end{align}	
	
		For each $i$, choose $c_{i}\in F^{\times}$ such that
	$c_{i}\mathfrak{a}_{i}$ is integral. Applying
	\cite[Chapter~I, \S5, Proposition~(5.2)]{neukirch_algebraic_1999}
	to the $\mathbb{Z}$-basis
	$\{c_{i}\alpha_{i\ell}\}_{1\leq\ell\leq d}$ of
	$c_{i}\mathfrak{a}_{i}$, we obtain
	\begin{align*}
		\nr_{F/\mathbb{Q}}(c_{i})^{2}
		(\det\Omega_{i})^{2}
		&=
		\det\left(
		\bigl(\sigma_{k}(c_{i}\alpha_{i\ell})\bigr)_
		{1\leq k,\ell\leq d}
		\right)^{2}=
		\Delta_{F}\nr(c_{i}\mathfrak{a}_{i})^{2}=
		\Delta_{F}
		|\nr_{F/\mathbb{Q}}(c_{i})|^{2}
		\nr(\mathfrak{a}_{i})^{2}.
	\end{align*}
	Since $c_{i}\neq0$, it follows that $
		(\det\Omega_{i})^{2}
		=
		\Delta_{F}\nr(\mathfrak{a}_{i})^{2}$. 
	Since $C$ is block diagonal,
	\begin{align}\label{detC}
		(\det C)^{2}
		&=
		\prod_{i=1}^{n}(\det\Omega_{i})^{2}=
		\Delta_{F}^{n}
		\prod_{i=1}^{n}\nr(\mathfrak{a}_{i})^{2}.
	\end{align}
	
	We next compute $\det A^{(1)}$. For $1\le k\le d-1$, put $
		m_{k}=
		d-k+1$ and $
		I_{k} =
		\{1,m_{k}+1,\ldots,(n-1)m_{k}+1\}$. For $1\le i,j\le n$ and $1\le r\le m_{k}$, the diagonal matrix $D_{ij}^{(k)}$ gives
	\begin{align*}
		\bigl(A^{(k)}\bigr)_
		{(i-1)m_{k}+1,\,(j-1)m_{k}+r}
		&=
		\begin{cases}
			\sigma_{k}(b_{ij}) & \text{if } r=1,\\
			0                  & \text{if } 2\le r\le m_{k}.
		\end{cases}
	\end{align*}
	Expanding $\det A^{(k)}$ along the rows indexed by $I_{k}$, we see that the only nonzero term in the Laplace expansion is the one for which the column set equals $I_{k}$. The corresponding minor is
	\begin{align*}
		\det\left(
		\bigl(\sigma_{k}(b_{ij})\bigr)_{1\le i,j\le n}
		\right)
		=
		\det\sigma_{k}(G_{L}),
	\end{align*}
	and its complementary minor is $\det A^{(k+1)}$. The sign in Laplace's expansion is 
	\begin{align*}
		(-1)^{
			\sum_{s\in I_{k}}s+
			\sum_{s\in I_{k}}s
		}
		=
		1,
	\end{align*}
	so no sign factor appears. Hence $
		\det A^{(k)}
		=
		\det\sigma_{k}(G_{L})
		\det A^{(k+1)}$ for $1\le k\le d-1$. 
	Since $A^{(d)}=\sigma_{d}(G_{L})$, we obtain
	\begin{align}\label{detA1}
		\det A^{(1)}
		&=
		\prod_{k=1}^{d}\det\sigma_{k}(G_{L})=
		\prod_{k=1}^{d}\sigma_{k}(\det G_{L})=
		\nr_{F/\mathbb{Q}}(\det G_{L}).
	\end{align}
	By the definition of the volume ideal (cf. \cite[\S 82E]{omeara_quadratic_1963}),
	\begin{align*}
		\mathfrak{v}(L)
		=
		\left(
		\prod_{i=1}^{n}\mathfrak{a}_{i}
		\right)^{2}
		(\det G_{L}).
	\end{align*}
	Since $L$ is positive definite, $\sigma_{k}(G_{L})$ is positive definite for each $k$. Therefore,
	\begin{align*}
		\nr_{F/\mathbb{Q}}(\det G_{L})
		=
		\prod_{k=1}^{d}\det\sigma_{k}(G_{L})
		>
		0,
	\end{align*}
	and hence
	\begin{align}\label{NvL}
		\nr(\mathfrak{v}(L))
		=
		\left(
		\prod_{i=1}^{n}\nr(\mathfrak{a}_{i})^{2}
		\right)
		\nr_{F/\mathbb{Q}}(\det G_{L}).
	\end{align}
	Combining \eqref{detG}, \eqref{detC}, \eqref{detA1}, and \eqref{NvL}, we obtain
	\begin{align*}
		\det G_{\Lambda}
		&=
		\Delta_{F}^{n}
		\left(
		\prod_{i=1}^{n}\nr(\mathfrak{a}_{i})^{2}
		\right)
		\nr_{F/\mathbb{Q}}(\det G_{L})=
		\Delta_{F}^{n}\nr(\mathfrak{v}(L)).
	\end{align*}
\end{proof}
\begin{lem}\label{lem:covolume-formula}
	Let $L$ be an $\mathcal{O}_F$-lattice of rank $n$ and $\Lambda=\sigma_V(L)$. Then
	\begin{align*}
		\vol( \Lambda) = \Delta_F^{n/2} \sqrt{\nr(\mathfrak{v}(L))}.
	\end{align*} 
\end{lem}
\begin{proof}
	This follows immediately from \eqref{vol-defn} and Lemma \ref{lem:gram-det}.
\end{proof}
\begin{lem}\label{lem:covering-volume-inequality}
	Let $L$ be an $\mathcal{O}_F$-lattice of rank $n$ and $\Lambda=\sigma(L)$. Put $U_m=\pi^{m/2}/\Gamma(m/2+1)$, where $\Gamma$ denotes the Gamma function. Then
	\begin{align*}
		\vol(\Lambda) \le U_{nd} \bigl(\rho_T(L)\bigr)^{nd/2}.
	\end{align*}  
\end{lem}
\begin{proof}
	For $r>0$, we denote by
	\begin{align*}
		B^{nd}(r) = \{z\in V_{\mathbb{R}} \mid \|z\|_{T} \le r \}
	\end{align*}
	the closed ball of radius $r$ with respect to the trace norm $\|\cdot\|_T$. The covering radius of $\Lambda$ in this norm is $R=\sqrt{\rho_T(L)}$. By the definition of the Voronoi cell and \eqref{rhoTL=supinf},
		\begin{align*}
		\|z\|_T
		&=
		\inf_{\lambda\in\Lambda}\|z-\lambda\|_{T}
		\le 
		\sqrt{\rho_T(L)}
		=
		R
	\end{align*}
for $z \in \mathcal{V}$. Hence $\mathcal{V} \subseteq B^{nd}(R)$, so by \eqref{eq:voronoi-covolume} and \cite[Theorem 7.30]{hoffstein_introduction_2014}, we deduce that
\begin{align*}
	\vol(\Lambda) = \vol(\mathcal{V}) \le \vol(B^{nd}(R)) = U_{nd}\, R^{nd} = U_{nd}\, \bigl(\rho_T(L)\bigr)^{nd/2}.
\end{align*} 
\end{proof}
\begin{lem}\label{lem:discriminant-norm-bound}
	Let $L$ be a trace Euclidean $\mathcal{O}_{F}$-lattice of rank $n$. Then
	\begin{align}\label{eq:disc-bound-step1}
		\nr(\mathfrak{v}(L))<
		\frac{(\pi d)^{nd}}{\Gamma(nd/2+1)^{2}\Delta_{F}^{n}}.
	\end{align}
	Moreover,
	\begin{align}\label{norm-inclusion}
		2^{-nd}\nr(\mathfrak{n}(L))^{n}
		\le\nr(\mathfrak{s}(L))^{n}
		\le\nr(\mathfrak{v}(L)).
	\end{align}
\end{lem}
\begin{proof}
	By Lemma \ref{lem:phi-Euclidean-trace}, $\rho_{T}(L)< d$. Combining Lemmas \ref{lem:covolume-formula} and \ref{lem:covering-volume-inequality}, we obtain
	\begin{align*}
		\Delta_{F}^{n/2}\sqrt{\nr(\mathfrak{v}(L))}
		=\vol(\Lambda)
		\le U_{nd}\bigl(\rho_{T}(L)\bigr)^{nd/2}
		<\frac{(\pi d)^{nd/2}}{\Gamma(nd/2+1)}.
	\end{align*}
	Rearranging the inequality gives \eqref{eq:disc-bound-step1}.
	Recall from \cite[\S~82E and 82:10]{omeara_quadratic_1963} that
	\begin{align}\label{norm-scale-volume-inclusion}
		\mathfrak{v}(L)\subseteq\mathfrak{s}(L)^{n}
		\subseteq2^{-n}\mathfrak{n}(L)^{n}.
	\end{align}
	Taking absolute norms gives \eqref{norm-inclusion}.
\end{proof}

\section{Bounds on root discriminant and the norm of volume ideal}
\label{sec:root-discriminant-bounds}

Throughout this section, we assume that $F$ is a totally real number field of degree $d=[F:\mathbb{Q}]$, and denote by  $\delta_{F}=\Delta_{F}^{1/d}$ the root discriminant of $F$. Let $L$ be a positive definite integral $\mathcal{O}_{F}$-lattice.

\begin{lem}\label{lem:root-discriminant-bound}
	Suppose that $L$ is trace Euclidean.  Put 
	\begin{align*}
		c=\begin{cases}
			1 &\text{if $L$ is classic integral or $n=1$},\\
			2 &\text{otherwise}.
		\end{cases}
	\end{align*}
Then
	\begin{align*} 
		\delta_{F}\le\frac{c\pi d}{\Gamma(nd/2+1)^{2/(nd)}}
		<\frac{2c\pi e}{n}.
	\end{align*}
	In particular, $\delta_{F}<2\pi e$ in all cases.
\end{lem}
\begin{proof}
	For the first assertion, note that every integral lattice of rank one is classic integral. So if $L$ is classic integral or $n=1$, then $ \mathfrak{s}(L)\subseteq \mathcal{O}_{F}$ and thus $\nr(\mathfrak{s}(L))\ge 1$; if $L$ is integral and $n\ge 2$, then $ \mathfrak{n}(L)\subseteq \mathcal{O}_{F}$ and thus $\nr(\mathfrak{n}(L))\ge 1$. By \eqref{norm-inclusion}, $\nr(\mathfrak{v}(L))\ge c^{-nd}$. Combining this with \eqref{eq:disc-bound-step1}, we have
	\begin{align*}
		\delta_{F}c^{-1}\le \Delta_{F}^{1/d}\,\nr(\mathfrak{v}(L))^{1/(nd)}
		\le\frac{\pi d}{\Gamma(nd/2+1)^{2/(nd)}}.
	\end{align*}
	This proves the first inequality.
	By \cite[Theorem~1.5]{batir_gammafunction_2008}, we have
	\begin{align*} 
		\Gamma(x+1)
		&\ge\sqrt{2e}\left(\frac{x+1/2}{e}\right)^{x+1/2}=\sqrt{2x+1}\left(\frac{x+1/2}{e}\right)^{x}
		>\left(\frac{x}{e}\right)^{x}
		\qquad(x>0).
	\end{align*}
	Taking $x=nd/2$ gives $\Gamma(nd/2+1)^{2/(nd)}>nd/(2e)$. This shows the second inequality.
	
	For the second assertion, if $n=1$, then $c=1$; if $n\ge 2$, then $2c/n\le 2$. So $\delta_{F}<2\pi e$.
\end{proof}
For $1\le d\le9$, we denote by $\Delta_{d}^{\min}$ the minimal discriminant
among totally real number fields of degree $d$.
From \cite[Table 3]{voight_totally-real_2008} and the fact that $\Delta_{1}^{\min}=1$, we have
\begin{align*}
	(\Delta_{1}^{\min},\ldots,\Delta_{9}^{\min})
	=(1,5,49,725,14641,300125,20134393,282300416,9685993193).
\end{align*}
\begin{lem}\label{lem:discriminant-lower-bounds}
	Put
	\begin{align}\label{eq:beta-discriminant}
		\beta_{d}\coloneq
		\begin{cases}
			(\Delta_{d}^{\min})^{1/d} &1\le d\le 9,\\
			14 &d=10,\\
			14.083 &d=11,\\
			36.347e^{-10.667/d} &d\ge 12.
		\end{cases}
	\end{align}
	Then $\delta_{F}\ge\beta_{d}$ for $d\ge1$. Moreover, if $\delta_{F}<2\pi e$, then $d\le14$.
\end{lem}
\begin{proof}
	Using the unconditional bound in \cite[Table~4 and description of tables]{odlyzko_discriminant-tables_1976} with $b=4.000$, $r_{1}=d$, $r_{2}=0$, $A=36.347$, and $E=10.667$, we obtain $\Delta_{F}>A^{d}e^{-E}$. It follows that
	\begin{align}\label{eq:odlyzko-explicit}
		\delta_{F}=\Delta_{F}^{1/d}>Ae^{-E/d}=36.347e^{-10.667/d}.
	\end{align}
	
	For $1\le d\le9$, by the definition of $\Delta_{d}^{\min}$, we have
	\begin{align*}
		\delta_{F}=\Delta_{F}^{1/d}
		\ge(\Delta_{d}^{\min})^{1/d}
		=\beta_{d}.
	\end{align*}

	From \cite[Table 3 and \S 1.1]{voight_totally-real_2008}, we know that $\delta_{F}>14$ if $d=10$ and $\delta_{F}>14.083$ if $d\ge11$. Combining these bounds with \eqref{eq:odlyzko-explicit} and the bounds for $1\le d\le9$ above, we obtain $\delta_{F}\ge\beta_{d}$ for all $d\ge 1$. 
	
	If $\delta_{F}<2\pi e$, then by \eqref{eq:odlyzko-explicit},
	\begin{align*}
		d<\frac{10.667}{\log\bigl(36.347/(2\pi e)\bigr)}<15
	\end{align*}
	and thus $d\le14$.
\end{proof}

\begin{cor}\label{cor:finite-grid}
	If $L$ is trace Euclidean, then $
		1\le n\le 34 $ and $1\le d\le 14$.
\end{cor}
\begin{proof}
	Since $\delta_{F}\ge1$, it follows from Lemma \ref{lem:root-discriminant-bound} that $n<4\pi e<35$ and $\delta_{F}<2\pi e$. Hence, by Lemma \ref{lem:discriminant-lower-bounds}, $d\le14$.
\end{proof}

For the remainder of this section, suppose that $L$ is trace Euclidean. By Lemma \ref{lem:discriminant-lower-bounds}, we have $\Delta_{F}^{1/d}=\delta_{F}\ge\beta_{d}$. So, by  \eqref{eq:disc-bound-step1}, we further have
\begin{align}\label{eq:Hd}
	\nr(\mathfrak{v}(L))\le H(n,d)
	\coloneq\frac{(\pi d)^{nd}}{\Gamma(nd/2+1)^{2}\beta_{d}^{nd}}.
\end{align}
Define
\begin{align*}
	h_{\mathfrak{s}}(n,d)\coloneq\frac1n\log H(n,d)\quad\text{and}\quad
	h_{\mathfrak{n}}(n,d)\coloneq d\log2+\frac1n\log H(n,d).
\end{align*}
From \eqref{norm-inclusion} and \eqref{eq:Hd}, 
\begin{equation}\label{eq:def-h}
	\begin{aligned}
		 &\log \nr(\mathfrak{s}(L))\le \log( \nr(\mathfrak{v}(L))^{1/n})\le  \dfrac{1}{n}\log H(n,d)= h_{\mathfrak{s}}(n,d),\\
		&\log \nr(\mathfrak{n}(L))\le  \log (2^{d}\nr(\mathfrak{v}(L))^{1/n})\le   d\log2+\frac1n\log H(n,d)=h_{\mathfrak{n}}(n,d).
	\end{aligned}
\end{equation}
So $h_{\mathfrak{s}}(n,d)\ge0$ in the classic integral case and $h_{\mathfrak{n}}(n,d)\ge0$ in the integral case. Therefore,
\begin{align}\label{eq:admissible-pairs}
	H(n,d)\ge
	\begin{cases}
		1 &\text{if $L$ is classic integral or $n=1$},\\
		2^{-nd} &\text{otherwise}.
	\end{cases}
\end{align}

\section{Finiteness theorems for trace Euclidean lattices}
\label{sec:finiteness}

In this section, we use the bounds from Section \ref{sec:root-discriminant-bounds} to prove the finiteness theorems.  

\begin{lem}\label{lem:bounded-discriminant-volume-finiteness}
	Fix positive integers $n, d$, and positive real numbers
	$D, C$. There are only finitely many isomorphism classes of totally real fields
	$F$ of degree $d$ with $\Delta_{F}\le D$ that admit a
	positive definite classic integral (resp. integral)
	$\mathcal{O}_{F}$-lattice $L$ of rank $n$ satisfying
	\begin{align*}
		\nr(\mathfrak{v}(L))\le C.
	\end{align*}
	Moreover, for each such fixed field $F$, there are only
	finitely many isometry classes of such lattices.
\end{lem}
\begin{proof}
	Since $\Delta_{F}$ is a positive integer, by 
	\cite[Chapter III, \S2, Theorem (2.16)]
	{neukirch_algebraic_1999}, there are only finitely many
	isomorphism classes of fields $F$ such that
	$[F:\mathbb{Q}]=d$ and $\Delta_{F}\le D$. Fix one such $F$.
	
	From \cite[Remark 103:5]{omeara_quadratic_1963},
	for every integral ideal $\mathfrak{a}$ of $F$, there are
	only finitely many isometry classes of $n$-dimensional
	quadratic spaces over $F$ which support an
	$\mathcal{O}_{F}$-lattice of integral scale and volume
	$\mathfrak{a}$. On each such space, from
	\cite[Theorem~103:4]{omeara_quadratic_1963} there are only finitely many proper isometry classes, and hence
	only finitely many isometry classes, of these lattices.
	
	If $L$ is classic integral, then, by
	\eqref{norm-scale-volume-inclusion},
	\begin{align*}
		\mathfrak{v}(L)
		\subseteq
		\mathfrak{s}(L)^{n}
		\subseteq
		\mathcal{O}_{F}.
	\end{align*}
	Thus $\mathfrak{v}(L)$ is an integral ideal and
	$\nr(\mathfrak{v}(L))\le C$. Since there are only finitely many
	integral ideals of norm at most $C$, the assertion follows by the finiteness results above in this case.
	
	If $L$ is integral, then, by \cite[\S82J]
	{omeara_quadratic_1963} and
	\eqref{norm-scale-volume-inclusion},
	\begin{align*}
		\mathfrak{s}(L^{(2)})
		&=
		2\mathfrak{s}(L)
		\subseteq
		\mathfrak{n}(L)
		\subseteq
		\mathcal{O}_{F}\quad \text{and}\quad
		\mathfrak{v}(L^{(2)})=
		2^{n}\mathfrak{v}(L)
		\subseteq
		\mathfrak{n}(L)^{n}
		\subseteq
		\mathcal{O}_{F},
	\end{align*}
	where $L^{(2)}$ denotes the scaling of $L$ by $2$. 
	Therefore, $L^{(2)}$ has integral scale and integral volume, with
	\begin{align*}
		\nr\bigl(\mathfrak{v}(L^{(2)})\bigr)
		=
		2^{nd}\nr(\mathfrak{v}(L))
		\le
		2^{nd}C.
	\end{align*} 
	Since there are only finitely many integral ideals of norm at most
	$2^{nd}C$, the preceding finiteness result applies to $L^{(2)}$. Since $L^{(2)}\cong M^{(2)}$ if and only if $L\cong M$, scaling by $2$ induces a bijection on the corresponding isometry
	classes, and the assertion follows for $L$.
\end{proof}
\begin{proof}[Proof of Theorems \ref{thm:finiteness-classic} and \ref{thm:finiteness-integral}]
	Let $L$ be a positive definite trace Euclidean $\mathcal{O}_{F}$-lattice as in either theorem. Recall that $\delta_{F}=\Delta_{F}^{1/d}$ and the bounds as asserted follow from Lemma \ref{lem:root-discriminant-bound}. By Corollary \ref{cor:finite-grid}, only finitely many pairs $(n,d)$ can occur.
	
	For each such pair, by Lemma \ref{lem:root-discriminant-bound}, we have $\Delta_{F}=\delta_{F}^{d}<(2\pi e)^{d}$. Since $\Delta_{F}\ge1$, by Lemma \ref{lem:discriminant-norm-bound}, we further have
	\begin{align*}
		\nr(\mathfrak{v}(L))\le
		\frac{(\pi d)^{nd}}{\Gamma(nd/2+1)^{2}\Delta_{F}^{n}}
		\le\frac{(\pi d)^{nd}}{\Gamma(nd/2+1)^{2}}.
	\end{align*}
	Applying Lemma \ref{lem:bounded-discriminant-volume-finiteness} with
	\begin{align*}
		D=(2\pi e)^{d}\quad\text{and}\quad
		C=\frac{(\pi d)^{nd}}{\Gamma(nd/2+1)^{2}},
	\end{align*}
	we obtain the asserted finiteness results.
\end{proof}

\section{Rank one trace Euclidean lattices over real quadratic fields}
\label{sec:trace-euclidean-fields}

Before discussing the classification, we explain the relation with
thin ideal classes. Let $F$ be totally real of degree $d$, and let
$L=(\mathfrak{a},ax^{2})$ be an integral $\mathcal{O}_{F}$-lattice
of rank one. Suppose that $L$ is trace Euclidean, and put
$n=\nr(a\mathfrak{a}^{2})$. Then $n\in\mathbb{Z}_{>0}$.
By Lemmas \ref{lem:gram-det} and \ref{lem:phi-Euclidean-trace},
the trace lattice has determinant $\Delta_{F}n$ and squared
covering radius $\rho_{T}(L)<d$. In the notation of \cite[\S\S1--2 and Lemma 4.3(ii)]{bayer-fluckiger_euclidean_2006}, we have
\begin{align*}
	\tau_{\min}(\mathfrak{a})
	&\le \frac{\rho_T(L)}{(\Delta_{F}n)^{1/d}}
	< \frac{d}{(\Delta_{F}n)^{1/d}}
	\le \frac{d}{\Delta_{F}^{1/d}}
	=\gamma_{\min}(\mathcal{O}_{F}).
\end{align*}
Thus $\mathfrak{a}$ is thin in the sense of \cite[Definition~12.1]{bayer-fluckiger_euclidean_2006}. So the finiteness of the ground fields in rank one also follows from \cite[Proposition~12.4]{bayer-fluckiger_euclidean_2006}. 

In the remainder of this section, we assume that $F$ is a real quadratic field. Take $V=F$ with $Q(x)=ax^{2}$,
where $a\in F$ is totally positive. Let $\mathfrak{a}$ be a nonzero fractional
ideal of $F$ such that $a\mathfrak{a}^{2}\subseteq\mathcal{O}_{F}$, and let
$L=(\mathfrak{a},Q)$. Every positive definite integral lattice of rank one
over $F$ has this form after choosing an $F$-basis of its quadratic space.
By Lemma \ref{lem:Phi}(v), we have
\begin{align}\label{eq:trace-covering-maximum}
	\rho_{a}(\mathfrak{a})
	\coloneq\rho_{T}(L)
	=\max_{x\in F}\inf_{y\in\mathfrak{a}}
	\tr_{F/\mathbb{Q}}\bigl(a(x-y)^{2}\bigr).
\end{align}
When $\mathfrak{a}=\mathcal{O}_{F}$ and $a=1$, we write
$\rho(\mathcal{O}_{F})=\rho_{1}(\mathcal{O}_{F})$.

\begin{prop}
	\label{prop:trace-euclidean-quadratic}
	Let $F=\mathbb{Q}(\sqrt m)$ be a real quadratic field with
	square-free $m>1$. Let $\mathfrak{a}$ be a nonzero fractional ideal of $F$,
	and let $a\in F$ be totally positive with
	$a\mathfrak{a}^{2}\subseteq\mathcal{O}_{F}$.
	There exists a $\mathbb{Z}$-basis $\{\beta_{1},\beta_{2}\}$ of
	$\mathfrak{a}$ for which the trace Gram matrix satisfies
	\begin{align*}
		G=\begin{pmatrix}
			\tr_{F/\mathbb{Q}}(a\beta_{1}^{2})&\tr_{F/\mathbb{Q}}(a\beta_{1}\beta_{2})\\
			\tr_{F/\mathbb{Q}}(a\beta_{1}\beta_{2})&\tr_{F/\mathbb{Q}}(a\beta_{2}^{2})
		\end{pmatrix}
		=\begin{pmatrix}A&B\\B&C\end{pmatrix},
		\quad 0<A\le C,\quad 2|B|\le A.
	\end{align*}
	Moreover, we have $	\det G=AC-B^{2}=\Delta_{F}\nr_{F/\mathbb{Q}}(a)\nr(\mathfrak{a})^{2}$ and
	\begin{align*}
	    \rho_{a}(\mathfrak{a})&=\frac{AC(A+C-2|B|)}{4(AC-B^{2})}.
	\end{align*}
	In particular, when $\mathfrak{a}=\mathcal{O}_{F}$ and $a\in\mathbb{Z}_{>0}$, then
	\begin{enumerate}[label=\rm(\roman*), itemindent=0em, leftmargin=*]
		\item If $m\equiv 2,3\pmod{4}$, then $\rho_{a}(\mathcal O_F)=a(m+1)/2$.
		
		\item If $m\equiv 1\pmod{4}$, then $\rho_{a}(\mathcal O_F)=a(m+1)^2/(8m)$.
	\end{enumerate}
\end{prop}
\begin{proof}
	Since $a\mathfrak{a}^{2}\subseteq\mathcal{O}_{F}$, the trace
	pairing $\tr_{F/\mathbb{Q}}(axy)$ on $\mathfrak{a}$ is
	positive definite and integral. By Gauss reduction, there
	exists a $\mathbb{Z}$-basis of $\mathfrak{a}$ whose trace
	Gram matrix satisfies
	\begin{align*}
		G=\begin{pmatrix}A&B\\B&C\end{pmatrix},
		\qquad A,B,C\in\mathbb{Z},\quad
		0<A\le C,\quad 2|B|\le A.
	\end{align*}
	Fix any such basis $\{\beta_{1},\beta_{2}\}$, and put
	\begin{align*}
		q(x,y)
		=\tr_{F/\mathbb{Q}}\bigl(a(x\beta_{1}+y\beta_{2})^{2}\bigr)
		=Ax^{2}+2Bxy+Cy^{2}.
	\end{align*}
	Since $\mathfrak{v}(L)=a\mathfrak{a}^{2}$,
	Lemma \ref{lem:gram-det} implies that
	\begin{align*}
		\det G=AC-B^{2}
		=\Delta_{F}\nr_{F/\mathbb{Q}}(a)\nr(\mathfrak{a})^{2}.
	\end{align*}
	Since $F=\mathbb{Q}\beta_{1}+\mathbb{Q}\beta_{2}$ and
	$\mathfrak{a}=\mathbb{Z}\beta_{1}+\mathbb{Z}\beta_{2}$, by
	\eqref{eq:trace-covering-maximum}, we have
	\begin{align*}
		E(q)\coloneq\sup_{z\in\mathbb{Q}^{2}}
		\inf_{\lambda\in\mathbb{Z}^{2}}q(z-\lambda)
		=\rho_{a}(\mathfrak{a}).
	\end{align*}
	Put $
		q_{+}(x,y)=Ax^{2}+2|B|xy+Cy^{2}$. One can replace $y$ by $-y$ if $B<0$  and so $E(q_{+})=E(q)$.
	Since $0\le2|B|\le A\le C$, applying
	\cite[Theorem~5.1(b)]{clark_ADC-II-2014} to $q_{+}$, we deduce that
	\begin{align}\label{rhoa}
		\rho_{a}(\mathfrak{a})
		=E(q)=E(q_{+})
		=\frac{AC(A+C-2|B|)}{4(AC-B^{2})}.
	\end{align}
	 For the last two assertions, take $\beta_{1}=1$ and
	 $\beta_{2}=\omega$, where
	 \begin{align}\label{omega}
	 	\omega=
	 	\begin{cases}
	 		\sqrt{m}
	 		&\text{if $m\equiv2,3\pmod{4}$},\\
	 		(1+\sqrt{m})/2
	 		&\text{if $m\equiv1\pmod{4}$}.
	 	\end{cases}
	 \end{align}
	 Recall from \cite[Theorem 5.4.2]{alaca_williams_2004} that
	 these form an integral basis of $\mathcal{O}_{F}$.
	 For $a\in\mathbb{Z}_{>0}$, the trace Gram matrices are
	 \begin{align*}
	 	G=
	 	\begin{cases}
	 		a\begin{pmatrix}2&0\\0&2m\end{pmatrix}
	 		&\text{if $m\equiv2,3\pmod{4}$},\\[6pt]
	 		a\begin{pmatrix}2&1\\1&(m+1)/2\end{pmatrix}
	 		&\text{if $m\equiv1\pmod{4}$}.
	 	\end{cases}
	 \end{align*}
	 In both cases, $0<A\le C$ and $2B\le A$.
	 Plugging into \eqref{rhoa} above gives the desired formulas.
\end{proof}
\begin{proof}[Proof of Theorem \ref{thm:rank-one-classification}]
	Write $L=(\mathfrak{a},ax^{2})$, where $a\in F$ is totally positive and
	$\mathfrak{a}$ is a nonzero fractional ideal with
	$a\mathfrak{a}^{2}\subseteq\mathcal{O}_{F}$.
	Choose a $\mathbb{Z}$-basis $\{\beta_{1},\beta_{2}\}$ as in
	Proposition \ref{prop:trace-euclidean-quadratic}, changing one sign
	if necessary so that the trace Gram matrix satisfies
	\begin{align*}
		G=\begin{pmatrix}A&B\\B&C\end{pmatrix},
		\quad A,B,C\in\mathbb{Z},\quad 0<A\le C,\quad 0\le2B\le A.
	\end{align*}
	By Lemma \ref{lem:phi-Euclidean-trace}, we have
	\begin{align}\label{eq:rank-one-strict-criterion}
		L\text{ is trace Euclidean}
		\quad\Longleftrightarrow\quad
		\rho_T(L)<2.
	\end{align}
	Suppose that $\rho_T(L)<2$. Put
	$n=\nr(a\mathfrak{a}^{2})\in\mathbb{Z}_{>0}$. By
	Proposition \ref{prop:trace-euclidean-quadratic}, we have
	\begin{align*}
		D=\det G=AC-B^{2}
		=\Delta_F\nr_{F/\mathbb{Q}}(a)\nr(\mathfrak{a})^{2}
		=\Delta_F n.
	\end{align*}
	Since $a\beta_1^{2}$ is a totally positive algebraic integer,
	\begin{align*}
		A=\tr_{F/\mathbb{Q}}(a\beta_1^{2})
		\ge2\sqrt{\nr_{F/\mathbb{Q}}(a\beta_1^{2})}\ge2.
	\end{align*}
	By \eqref{rhoa}, we have
	\begin{align*}
		\rho_T(L)-\frac C4
		=\dfrac{C(A-B)^{2}}{4(AC-B^{2})}=\dfrac{C(A-B)^{2}}{4D}\ge0.
	\end{align*}
	It follows that
	\begin{equation}\label{eq:rank-one-gram-bounds}
		\begin{aligned}
			2\le A\le C\le7,\quad 0\le2B\le A,\quad
			AC(A+C-2B)<8(AC-B^{2}).
		\end{aligned}
	\end{equation}
	The integer triples satisfying \eqref{eq:rank-one-gram-bounds}
	are given by
	\begin{align*}
		(A,B,C)\in\bigl\{&(2,0,2),(2,0,3),(2,0,4),(2,0,5),
		(2,1,2),(2,1,3),(2,1,4),(2,1,5),\\
		&(2,1,6),(2,1,7),(3,0,3),(3,0,4),
		(3,1,3),(3,1,4),(3,1,5),(3,1,6),\\
		&(4,1,4),(4,1,5),(4,2,4),(4,2,5),
		(4,2,6),(5,2,5)\bigr\}.
	\end{align*}
	These $22$ triples all satisfy $D\le21$. Since $D=\Delta_F n$,
	we have $\Delta_F\le21$. For square-free $m>1$, recall that
	\begin{align*}
		\Delta_F=
		\begin{cases}
			m,&m\equiv1\pmod4,\\
			4m,&m\equiv2,3\pmod4.
		\end{cases}
	\end{align*}
	Hence $\Delta_F\in\{5,8,12,13,17,21\}$ and thus
	$m\in\{2,3,5,13,17,21\}$.
	We retain only the triples for which
	$n=(AC-B^{2})/\Delta_F$ is a positive integer.
	When $m\equiv2,3\pmod4$, we also have
	$\tr_{F/\mathbb{Q}}(\mathcal{O}_{F})=2\mathbb{Z}$, so
	$A,B,C\in2\mathbb{Z}$.
	This leaves the following nine candidates,
	\begin{align*}
		(m,A,B,C,n)\in \{&(2,2,0,4,1),(3,4,2,4,1),(5,2,1,3,1),(5,2,0,5,2),(5,4,1,4,3),\\
		&(5,4,2,6,4),(13,2,1,7,1),(17,3,1,6,1),(21,5,2,5,1)\}.
	\end{align*}
	Put $b_i=a\beta_i^{2}$ and
	$k_i=[\mathfrak{a}:\beta_i\mathcal{O}_{F}]
	\in\mathbb{Z}_{>0}$ for $i=1,2$.
	Since $\beta_i\mathcal{O}_{F}\subseteq\mathfrak{a}$, we have
	\begin{align}\label{eq:rank-one-index-norm}
		\nr_{F/\mathbb{Q}}(b_i)
		=\nr_{F/\mathbb{Q}}(a)\nr_{F/\mathbb{Q}}(\beta_i)^{2}
		=nk_i^{2}.
	\end{align}
	Put $\theta=(\Delta_F+\sqrt{\Delta_F})/2$, so that
	$\mathcal{O}_{F}=\mathbb{Z}+\mathbb{Z}\theta$.
	Write $b_1=s+t\theta$ with $s,t\in\mathbb{Z}$.
	Since $A=\tr_{F/\mathbb{Q}}(b_1)=2s+t\Delta_F$, we obtain
	\begin{align}\label{eq:rank-one-trace-norm}
		b_1=\frac{A+t\sqrt{\Delta_F}}2,\quad
		A\equiv\Delta_Ft\pmod2,\quad
		A^{2}-\Delta_Ft^{2}=4nk_1^{2}.
	\end{align}
	In particular, $4nk_1^{2}\le A^{2}$.
	The candidate $(m,A,B,C,n)=(5,2,0,5,2)$ is therefore ruled out.
	For $(m,A,B,C,n)=(5,4,1,4,3)$, we must have $k_1=1$, and \eqref{eq:rank-one-trace-norm} gives $5t^{2}=4$,
	which is impossible.
	Similarly, the candidate $(m,A,B,C,n)=(17,3,1,6,1)$
	would give $k_1=1$ and $17t^{2}=5$, which is also impossible.
	For the remaining candidates with $m\ne3$,
	\eqref{eq:rank-one-trace-norm} gives
	\begin{align*}
		(b_1,k_1)=
		\begin{cases}
			(1,1),&m=2,13\text{ or }(m,n)=(5,1),\\
			(2,1),&(m,n)=(5,4),\\
			((5\pm\sqrt{21})/2,1),&m=21.
		\end{cases}
	\end{align*}
	In all these cases, we have $k_1=1$, so
	$\mathfrak{a}=\beta_1\mathcal{O}_{F}$.
	
	For $m=3$, we have $A=C=4$ and $n=1$.
	Applying \eqref{eq:rank-one-trace-norm} to each $b_{i}$ gives
	\begin{align*}
		(b_i,k_i)\in
		\{(2,2),(2+\sqrt3,1),(2-\sqrt3,1)\},
		\qquad i=1,2.
	\end{align*}
	If $b_1=b_2=2$, then $\beta_1^{2}=\beta_2^{2}$ and hence
	$\beta_1=\pm\beta_2$, which contradicts the choice of a basis.
	Thus $k_i=1$ for some $i$, and again
	$\mathfrak{a}=\beta_i\mathcal{O}_{F}$.
	In each case, the map $x\mapsto\beta_i x$ is an isometry from
	$(\mathcal{O}_{F},b_ix^{2})$ to $L$.
	In particular, $L$ is free.
	
	Let
	\begin{align*}
		b=\begin{cases}
			2+\sqrt{3} &\text{if $m=3$},\\
			(5+\sqrt{21})/2 &\text{if $m=21$}.
		\end{cases}
	\end{align*}
	Choose the real embedding $\sigma_{2}$ so that
	$\sigma_{2}(\sqrt{m})=-\sqrt{m}$, and put
	$\bar{b}=\sigma_{2}(b)$.
	Since $b\bar{b}=1$, both $b$ and $\bar{b}$ are units of $\mathcal{O}_{F}$,
	and $\bar{b}=b(\bar{b})^{2}$.
	Thus $\langle \bar{b}\rangle\cong\langle b\rangle$.
	This gives the six pairs in the theorem.
	
	Conversely, each listed coefficient $a$ is a totally positive
	algebraic integer, so the unary lattice $\langle a\rangle$ is positive definite
	and integral. For the six listed lattices, we use the following
	$\mathbb{Z}$-bases of $\mathcal{O}_{F}$:
	\begin{align*}
		(\beta_1,\beta_2)=
		\begin{cases}
			(1,\sqrt2),&(m,a)=(2,1),\\
			(1,\sqrt3-1),&(m,a)=(3,2+\sqrt3),\\
			(1,\omega),&(m,a)=(5,1),(5,2),(13,1),\\
			(3-\omega,1),&(m,a)=(21,(5+\sqrt{21})/2),
		\end{cases}
	\end{align*}
	where $\omega$ is defined in \eqref{omega}. Their trace Gram matrices are reduced, and by
	Proposition \ref{prop:trace-euclidean-quadratic}, we obtain the following table
	\begin{align*} 
		\begin{array}{c|c|c|c}
			\toprule
			m&a&(A,B,C)&\rho_T(\langle a\rangle)\\ 
			\midrule
			2&1&(2,0,4)&3/2\\
			\hline
			3&2+\sqrt3&(4,2,4)&4/3\\
			\hline
			5&1&(2,1,3)&9/10\\
			\hline
			5&2&(4,2,6)&9/5\\
			\hline
			13&1&(2,1,7)&49/26\\
			\hline
			21&(5+\sqrt{21})/2&(5,2,5)&25/14\\
			\bottomrule
		\end{array}
	\end{align*}
	All six radii are strictly less than $2$, so the converse
	follows from the equivalent condition \eqref{eq:rank-one-strict-criterion}.
	Furthermore, the two lattices over $\mathbb{Q}(\sqrt5)$ have
	volume ideals of norms $1$ and $4$, respectively, and hence
	are not isometric. The other isometry classes lie over
	distinct real quadratic fields.
\end{proof}
\begin{proof}[Proof of Corollary \ref{cor:trace-euclidean-quadratic}]
	Let $L=(\mathcal{O}_{F},x^{2})$. By Definition \ref{defn:trace-euclidean-field}, $F$ is trace Euclidean
	if and only if $L$ is trace Euclidean.
	If $F$ is trace Euclidean, Theorem \ref{thm:rank-one-classification}
	implies that $m\in \{2,3,5,13,21\}$.
	For these values of $m$, taking $a=1$ in
	Proposition \ref{prop:trace-euclidean-quadratic}, one can check that
	\begin{align*}
		\rho_T(L)<2
		\quad\Longleftrightarrow\quad
		m\in\{2,5,13\}.
	\end{align*}
	The corollary follows from \eqref{eq:rank-one-strict-criterion}.
\end{proof}
 
\clearpage
\begin{table}[H]
	\centering
	\caption{Admissible pairs $(n,d)$ in all ranks $n\ge1$ permitted by \eqref{eq:admissible-pairs} for classic integral trace Euclidean $\mathcal{O}_{F}$-lattices, together with approximate values of $h_{\mathfrak{s}}(n,d)$ (cf. \eqref{eq:def-h}) and upper bounds for $\nr(\mathfrak{v}(L))$ (cf. \eqref{eq:Hd}).}
	\label{tab:admissible-classical}
	\small
	\begin{tabular}{cccc}
		\toprule
		Degree $d$ & Rank $n$ & $h_{\mathfrak{s}}(n,d)$ & Upper bound for $\nr(\mathfrak{v}(L))$ \\
		\midrule
		\multirow{12}{*}{1}
		& 1 & 1.386294 & $4.00000$ \\
		& 2 & 1.144730 & $9.86961$ \\
		& 3 & 0.954941 & $17.5460$ \\
		& 4 & 0.798156 & $24.3523$ \\
		& 5 & 0.664340 & $27.7075$ \\
		& 6 & 0.547477 & $26.7053$ \\
		& 7 & 0.443662 & $22.3235$ \\
		& 8 & 0.350216 & $16.4732$ \\
		& 9 & 0.265216 & $10.8802$ \\
		& 10 & 0.187232 & $6.50334$ \\
		& 11 & 0.115173 & $3.54985$ \\
		& 12 & 0.048188 & $1.78293$ \\
		\midrule
		\multirow{5}{*}{2}
		& 1 & 2.066316 & $7.89569$ \\
		& 2 & 1.373169 & $15.5855$ \\
		& 3 & 0.871810 & $13.6731$ \\
		& 4 & 0.477289 & $6.74740$ \\
		& 5 & 0.151320 & $2.13102$ \\
		\midrule
		\multirow{3}{*}{3}
		& 1 & 2.268840 & $9.66819$ \\
		& 2 & 1.046447 & $8.10835$ \\
		& 3 & 0.199664 & $1.82029$ \\
		\midrule
		\multirow{2}{*}{4}
		& 1 & 2.151631 & $8.59888$ \\
		& 2 & 0.359872 & $2.05391$ \\
		\midrule
		5 & 1 & 1.777311 & $5.91394$ \\
		\midrule
		6 & 1 & 1.423463 & $4.15148$ \\
		\bottomrule
	\end{tabular}
\end{table}
\begin{table}[H]
	\centering
	\caption{\textbf{Classic integral case.} Upper bounds $N_{d}$ for the rank and $D_{n}$ for the degree of trace Euclidean $\mathcal{O}_{F}$-lattices, covering every positive degree and rank. }
	\label{tab:max-rank-degree-classical}
	\begin{tabular}{cc|cc}
		\toprule
		Degree $d$ & Rank bound $N_{d}$ & Rank $n$ & Degree bound $D_{n}$ \\
		\midrule
		$1$ & $12$ & $1$ & $6$ \\
		$2$ & $5$ & $2$ & $4$ \\
		$3$ & $3$ & $3$ & $3$ \\
		$4$ & $2$ & $4\le n\le5$ & $2$ \\
		$5\le d\le6$ & $1$ & $6\le n\le12$ & $1$ \\
		$d\ge7$ & $0$ & $n\ge13$ & $0$ \\
		\bottomrule
	\end{tabular}
\end{table}
\newpage
\begin{table}[H]
	\centering
	\caption{Admissible pairs $(n,d)$ in all ranks $n\ge1$ permitted by \eqref{eq:admissible-pairs} for integral trace Euclidean $\mathcal{O}_{F}$-lattices, together with approximate values of $h_{\mathfrak{n}}(n,d)$ (cf. \eqref{eq:def-h}) and upper bounds for $\nr(\mathfrak{v}(L))$ (cf. \eqref{eq:Hd}).}
	\label{tab:admissible-norm}
	\small
	\begin{tabular}{cccc}
		\toprule
		Degree $d$ & Rank $n$ & $h_{\mathfrak{n}}(n,d)$ & Upper bound for $\nr(\mathfrak{v}(L))$ \\
		\midrule
		\multirow{29}{*}{1}
		& 1 & 2.079442 & $4.00000$ \\
		& 2 & 1.837877 & $9.86961$ \\
		& 3 & 1.648088 & $17.5460$ \\
		& 4 & 1.491303 & $24.3523$ \\
		& 5 & 1.357488 & $27.7075$ \\
		& 6 & 1.240624 & $26.7053$ \\
		& 7 & 1.136809 & $22.3235$ \\
		& 8 & 1.043364 & $16.4732$ \\
		& 9 & 0.958363 & $10.8802$ \\
		& 10 & 0.880379 & $6.50334$ \\
		& 11 & 0.808320 & $3.54985$ \\
		& 12 & 0.741335 & $1.78293$ \\
		& 13 & 0.678744 & $0.829245$ \\
		& 14 & 0.619997 & $0.359118$ \\
		& 15 & 0.564641 & $0.145499$ \\
		& 16 & 0.512302 & $0.0553806$ \\
		& 17 & 0.462661 & $0.0198757$ \\
		& 18 & 0.415452 & $0.00674795$ \\
		& 19 & 0.370443 & $0.00217358$ \\
		& 20 & 0.327436 & $6.65996\times10^{-4}$ \\
		& 21 & 0.286258 & $1.94579\times10^{-4}$ \\
		& 22 & 0.246758 & $5.43233\times10^{-5}$ \\
		& 23 & 0.208803 & $1.45212\times10^{-5}$ \\
		& 24 & 0.172276 & $3.72326\times10^{-6}$ \\
		& 25 & 0.137071 & $9.17233\times10^{-7}$ \\
		& 26 & 0.103095 & $2.17439\times10^{-7}$ \\
		& 27 & 0.070265 & $4.96720\times10^{-8}$ \\
		& 28 & 0.038504 & $1.09492\times10^{-8}$ \\
		& 29 & 0.007745 & $2.33172\times10^{-9}$ \\
		\bottomrule
	\end{tabular}
\end{table}
\newpage
\begin{table}[H]
	\centering
	\ContinuedFloat
	\caption[]{(continued). The columns have the same meaning as in the preceding part.}
	\small
	\begin{tabular}{cccc}
		\toprule
		Degree $d$ & Rank $n$ & $h_{\mathfrak{n}}(n,d)$ & Upper bound for $\nr(\mathfrak{v}(L))$ \\
		\midrule
		\multirow{12}{*}{2}
		& 1 & 3.452611 & $7.89569$ \\
		& 2 & 2.759463 & $15.5855$ \\
		& 3 & 2.258104 & $13.6731$ \\
		& 4 & 1.863584 & $6.74740$ \\
		& 5 & 1.537614 & $2.13102$ \\
		& 6 & 1.259527 & $0.467384$ \\
		& 7 & 1.016850 & $0.0753125$ \\
		& 8 & 0.801460 & $0.00929131$ \\
		& 9 & 0.607760 & $9.05695\times10^{-4}$ \\
		& 10 & 0.431728 & $7.15108\times10^{-5}$ \\
		& 11 & 0.270373 & $4.66634\times10^{-6}$ \\
		& 12 & 0.121408 & $2.55861\times10^{-7}$ \\
		\midrule
		\multirow{7}{*}{3}
		& 1 & 4.348282 & $9.66819$ \\
		& 2 & 3.125888 & $8.10835$ \\
		& 3 & 2.279105 & $1.82029$ \\
		& 4 & 1.628022 & $0.164364$ \\
		& 5 & 1.097941 & $0.00739092$ \\
		& 6 & 0.650372 & $1.88877\times10^{-4}$ \\
		& 7 & 0.262790 & $3.00103\times10^{-6}$ \\
		\midrule
		\multirow{5}{*}{4}
		& 1 & 4.924220 & $8.59888$ \\
		& 2 & 3.132460 & $2.05391$ \\
		& 3 & 1.924347 & $0.0784946$ \\
		& 4 & 1.008213 & $8.60924\times10^{-4}$ \\
		& 5 & 0.268749 & $3.65579\times10^{-6}$ \\
		\midrule
		\multirow{4}{*}{5}
		& 1 & 5.243047 & $5.91394$ \\
		& 2 & 2.857502 & $0.296275$ \\
		& 3 & 1.278816 & $0.00141481$ \\
		& 4 & 0.092788 & $1.38226\times10^{-6}$ \\
		\midrule
		\multirow{3}{*}{6}
		& 1 & 5.582346 & $4.15148$ \\
		& 2 & 2.586614 & $0.0430868$ \\
		& 3 & 0.631313 & $2.53506\times10^{-5}$ \\
		\midrule
		7 & 2 & 1.143409 & $6.00804\times10^{-4}$ \\
		\midrule
		8 & 2 & 1.275464 & $1.95603\times10^{-4}$ \\
		\midrule
		9 & 2 & 0.520141 & $1.07957\times10^{-5}$ \\
		\bottomrule
	\end{tabular}
\end{table}
\begin{table}[H]
	\centering
	\caption{\textbf{Integral case.} Upper bounds $N_{d}$ for the rank and $D_{n}$ for the degree of trace Euclidean $\mathcal{O}_{F}$-lattices, covering every positive degree and rank.}
	\label{tab:max-rank-degree-norm}
	\begin{tabular}{cc|cc}
		\toprule
		Degree $d$ & Rank bound $N_{d}$ & Rank $n$ & Degree bound $D_{n}$ \\
		\midrule
		$1$ & $29$ & $1$ & $6$ \\
		$2$ & $12$ & $2$ & $9$ \\
		$3$ & $7$ & $3$ & $6$ \\
		$4$ & $5$ & $4$ & $5$ \\
		$5$ & $4$ & $5$ & $4$ \\
		$6$ & $3$ & $6\le n\le7$ & $3$ \\
		$7\le d\le9$ & $2$ & $8\le n\le12$ & $2$ \\
		$d\ge10$ & $0$ & $13\le n\le29$ & $1$ \\
		&  & $n\ge30$ & $0$ \\
		\bottomrule
	\end{tabular}
\end{table}
\section*{Acknowledgments}
This work was supported by the National Natural Science Foundation of China (Project No. 12301013 and No. 12671016).

\section*{AI Declaration}
ChatGPT (GPT-5.6 Sol) was used to assist with proofreading the manuscript and improving its exposition, with suggestions incorporated into the revised version. It was also employed to write Mathematica code for verifying the numerical calculations in Section \ref{sec:trace-euclidean-fields} and for computing the entries in Tables \ref{tab:admissible-classical}, \ref{tab:max-rank-degree-classical}, \ref{tab:admissible-norm}, and \ref{tab:max-rank-degree-norm}, and to develop the Lean 4 formalization for potential future use. The new notions introduced, the motivation presented, and the mathematical connections established in this paper originated with the author. All the mathematical arguments, Mathematica code, and calculations involving AI assistance were independently verified by the author, who takes full responsibility for the content of this paper.

The Lean 4 formalization, which covers selected main results, compiled without errors and passed the automated checks. The verification and formalization materials, including documentation of the scope and external assumptions of the Lean formalization, are available in the accompanying repository \cite{he_trace-euclidean-verification_2026}.

\end{document}